\documentclass[12pt]{amsart}
\usepackage[utf8]{inputenc}
\usepackage{amsaddr}
\usepackage{comment}
\usepackage{amsmath,amsfonts,amssymb,amsthm}
\usepackage{xcolor} 
\usepackage{bm}
\usepackage{enumitem}
\usepackage{mathrsfs}
\usepackage{multicol}
\usepackage{centernot}
\usepackage{mathtools}
\usepackage{ stmaryrd }
\usepackage{tikz-cd}
\usepackage{url}
\usepackage{hyperref}
\usepackage[margin = 1in]{geometry}

\newtheorem{thm}{Theorem}[section]

\newtheorem{cor}[thm]{Corollary}
\newtheorem{lemma}[thm]{Lemma}
\newtheorem{prop}[thm]{Proposition}

\theoremstyle{definition}
\newtheorem{defn}[thm]{Definition}
\newtheorem{example}[thm]{Example}
\newtheorem{remark}[thm]{Remark}
\newtheorem{non-example}[thm]{Non-Example}

\newtheorem{fact}[thm]{Fact}

\DeclareMathOperator{\End}{End}
\DeclareMathOperator{\Hom}{Hom}

\DeclareMathOperator{\Lie}{Lie}
\DeclareMathOperator{\diag}{diag}
\def\mf{\mathfrak}

\def\phi{\varphi}
\def\cong{\simeq}
 
\def\hat{\widehat}

\def\C{\mathbb{C}}

\title{Dual pairs in complex classical groups and Lie algebras}
\author{Marisa Gaetz}

\email{mgaetz@mit.edu}
\thanks{The author was supported by the NSF Graduate Research Fellowship Program under Grant Nos.~1745302 and 2141064, and by the Fannie \& John Hertz Foundation.}

\begin{document}

\maketitle

\begin{abstract}
In Roger Howe's 1989 paper, ``Remarks on classical invariant theory," Howe introduces the notion of a \textit{dual pair of Lie subalgebras} -- a pair $(\mathfrak{g}_1, \mathfrak{g}_2)$ of reductive Lie subalgebras of a Lie algebra $\mathfrak{g}$ such that $\mathfrak{g}_1$ and $\mathfrak{g}_2$ are each other's centralizers in $\mathfrak{g}$. This notion has a natural analog for algebraic groups: a \textit{dual pair of subgroups} is a pair $(G_1, G_2)$ of reductive subgroups of an algebraic group $G$ such that $G_1$ and $G_2$ are each other's centralizers in $G$. In this paper, we classify the dual pairs in the complex classical groups ($GL(n,\mathbb{C})$, $SL(n,\mathbb{C})$, $Sp(2n,\mathbb{C})$, $O(n,\mathbb{C})$, and $SO(n,\mathbb{C})$) and in the corresponding Lie algebras ($\mf{gl}(n,\C)$, $\mf{sl}(n,\C)$, $\mf{sp}(2n,\C)$, and $\mf{so}(n,\C)$). We also present substantial progress towards classifying the dual pairs in the projective counterparts of the complex classical groups ($PGL(n,\mathbb{C})$, $PSp(2n,\mathbb{C})$, $PO(n,\mathbb{C})$, and $PSO(n,\mathbb{C})$). 
\end{abstract}

\section{Introduction}

In Roger Howe's seminal 1989 paper \cite{HoweRemarks}, Howe introduced the notion of a \textit{reductive dual pair of Lie subalgebras}:

\begin{defn}[Howe {\cite[p.~550]{HoweRemarks}}]
Let $\mf{g}$ be a reductive Lie algebra. Then a pair $(\mf{a}, \mf{b})$ of subalgebras of $\mf{g}$ is a \textbf{reductive dual pair} in $\mf{g}$ if
\begin{enumerate}[label = (\arabic*)]
\item $\mf{a}$ and $\mf{b}$ are reductive in $\mf{g}$, and
\item $\mf{a} = \mf{z}_{\mf{g}} (\mf{b})$ and $\mf{b} = \mf{z}_{\mf{g}}(\mf{a})$,
\end{enumerate}
where $\mf{z}_{\mf{g}} (\cdot)$ denotes taking the centralizer in $\mf{g}$.
\end{defn}

This notion has a natural analog for algebraic groups, which was also informally introduced by Howe:
\begin{defn}[Howe \cite{HoweRemarks}]
A pair $(G_1, G_2)$ of reductive subgroups of an algebraic group $G$ form a \textbf{reductive dual pair} in $G$ if $Z_{G}(G_1) = G_2$ and $Z_{G}(G_2) = G_1$.
\end{defn}

We only consider \textit{reductive} dual pairs in this paper, so we will sometimes just use ``dual pair" to refer to a reductive dual pair.

\subsection{Lie algebra dual pairs}

Lie algebra dual pairs have been widely used and studied since their introduction by Howe in 1989. For example, some of the earliest work on such dual pairs was done by Rubenthaler in his 1994 paper \cite{rubenthaler}, in which he outlines a classification of the conjugacy classes of reductive dual pairs in complex reductive Lie algebras. In addition to Rubenthaler's classification, certain reductive dual pairs of the exceptional Lie algebras are implicitly used in some of the well-known constructions of these algebras: $(G_2^1, A_2^{2''})$ and $(G_2^1, F_4^1)$ correspond to the Freudenthal-Tits constructions of $E_6$ and $E_8$, respectively \cite{tits-construction}, and $(D_4, D_4)$ corresponds to the triality construction of $E_8$ \cite{allison, barton-sudbery}. The real forms of certain reductive dual pairs in exceptional Lie algebras have also drawn some attention \cite{KovT, Kov}. Additionally, as explained in \cite{Panyushev}, the study of reductive dual pairs has some beautiful connections to the study of commuting nilpotent elements in complex semisimple Lie algebras (as introduced by Ginzburg in \cite{Ginzburg}).

While Lie algebra dual pairs are reasonably well-understood, a clear and complete account of the reductive dual pairs in the complex classical Lie algebras is lacking. In particular, Rubenthaler's approach in \cite{rubenthaler} was designed with dual pairs of exceptional Lie algebras in mind, and is much more complicated than what is needed for classical Lie algebra dual pairs. In this paper, we fill this gap by providing classifications of the dual pairs in $\mf{gl}(n,\C)$, $\mf{sl}(n,\C)$, $\mf{sp}(2n,\C)$, and $\mf{so}(n,\C)$ using elementary methods.

\subsection{Lie group dual pairs}

Compared to Lie algebra dual pairs, much less is known about dual pairs in Lie groups, and many of the known results concern groups defined over $p$-adic fields \cite{Savin4, Savin2, Savin3, Savin1}. The earliest progress on the topic of Lie group dual pairs was done by Howe: In the same paper in which he introduced the notion of Lie group dual pairs \cite{HoweRemarks}, Howe classifies certain dual pairs in $Sp(2n,\C)$. In particular, he classifies the \textit{classical reductive dual pairs} in $Sp(2n,\C)$ (i.e.~the reductive dual pairs in which both members are themselves classical groups). In \cite{schmidt}, Schmidt expands on Howe's work by classifying the classical reductive dual pairs in the groups of isometries of real, complex, and quaternionic Hermitian spaces (including $O(n,\C)$ and $Sp(2n,\C)$). While full classifications of the reductive dual pairs in these groups (as well as classifications of dual pairs in $GL(n,\C)$ and $SO(n,\C)$) do not require the development of novel methodologies, an explicit exposition of these classifications is lacking in the literature. This paper fills this gap.

The study of Lie group dual pairs is further motivated by the following fact, which shows that \textit{any} centralizer is a member of a dual pair:

\begin{fact} \label{fact:triple centralizer}
Consider a group $G$ and a subset $S \subseteq G$. Then $Z_G(Z_G(Z_G(S))) = Z_G(S)$. Therefore, $( Z_G(S), Z_G(Z_G(S)) )$ is a dual pair in $G$.
\end{fact}

Consequently, a classification of dual pairs in $G$ gives a complete list of possible centralizers in $G$. In this paper, we will often be concerned with centralizers $G^{\Phi} := Z_G(\Phi (H) )$, where $\Phi \colon H \rightarrow G$ is a homomorphism of algebraic groups. We will frequently make use of the fact that all reductive dual pairs in $G$ arise in this way:
\begin{remark} \label{rmk:all pairs}
Let $G$ be a complex reductive algebraic group. Then any reductive dual pair $(G_1, G_2)$ in $G$ can be written in the form $(G^{\Phi}, Z_G(G^{\Phi}))$ for some algebraic group homomorphism $\Phi \colon H \rightarrow G$. Indeed, take $\Phi$ to be the inclusion $G_2 \hookrightarrow G$. We get $G_1 = Z_G(\Phi (G_2))= G^{\Phi}$ and $G_2 = Z_G(G_1) = Z_G(G^{\Phi})$.
\end{remark}

\subsection{Outline}

As mentioned above, this paper fills gaps in the literature by providing classifications of the dual pairs in the complex classical Lie groups and Lie algebras using elementary methods. In doing so, this paper will also discuss relationships between dual pairs in Lie groups and dual pairs in corresponding Lie subgroups, quotient groups, or Lie algebras. We will also present substantial progress towards classifying the dual pairs in the projective classical groups $PGL(n,\C)$, $PSp(2n,\C)$, $PO(n\C)$, and $PSO(n,\C)$. We proceed as follows:
\begin{itemize}
    \item \textbf{Section \ref{sec:subgroups}:} We discuss the relationship between dual pairs in complex reductive algebraic groups and in certain subgroups of these groups. 
    \item \textbf{Section \ref{sec:quotients}:} We discuss the relationship between dual pairs and complex reductive algebraic groups and in certain quotients of these groups. 
    \item \textbf{Section \ref{sec:dps-in-lie-algs}:} We discuss the relationship between dual pairs in complex reductive Lie groups and in their corresponding Lie algebras, as well as Rubenthaler's reduction of the classification of reductive dual pairs in reductive Lie algebras to the classification of semisimple dual pairs in simple Lie algebras.
    \item \textbf{Section \ref{sec:Embeddings}:} We define embeddings of pairs of complex classical groups into larger complex classical groups. 
    \item \textbf{Section \ref{sec:gl}:} We classify the dual pairs in $GL(n,\C)$ and $SL(n,\C)$, the connected dual pairs in $PGL(n,\C)$, and the dual pairs in $\mf{sl}(n,\C)$ and $\mf{gl}(n,\C)$.
    \item \textbf{Section \ref{sec:sp}:} We classify the dual pairs in $Sp(2n,\C)$, some of the dual pairs in $PSp(2n,\C)$, and the dual pairs in $\mf{sp}(2n,\C)$.
    \item \textbf{Section \ref{sec:o}:} We classify the dual pairs in $O(n,\C)$ and $SO(n,\C)$, some of the dual pairs in $PO(n,\C)$ and $PSO(n,\C)$, and the dual pairs in $\mf{so}(n,\C)$. 
\end{itemize}

\section{Dual pairs in subgroups} \label{sec:subgroups}

Let $G$ be a complex reductive algebraic group. In this section, we discuss how dual pairs in $G$ relate to dual pairs in certain subgroups of $G$. 

\begin{lemma} \label{lem:G1-H1}
Suppose that $G$ equals a product of subgroups $G = KH$, where $K$ is central. If $G_1$ is a subgroup of $G$ containing $K$, then $G_1=KH_1$, where $H_1 := G_1 \cap H$.
\end{lemma}

\begin{proof}
Certainly, $KH_1 \subseteq G_1$. On the other hand, let $g \in G_1$ and write $g = kh$, where $k \in K$ and $h \in H$. Since $K \subseteq G_1$, we see that $h = k^{-1}g \in G_1$, so $h \in G_1 \cap H = H_1$. It follows that $g \in KH_1$ and hence that $G_1 \subseteq KH_1$.
\end{proof}

\begin{lemma} \label{lem:restrict-to-subgroup}
Suppose that $G$ equals a product of subgroups $G = KH$, where $K$ is central. If $(G_1, G_2)$ is a dual pair in $G$, then $(G_1 \cap H, G_2 \cap H)$ is a dual pair in $H$.
\end{lemma}

\begin{proof}
Since $G_1$ and $G_2$ are centralizers in $G$, they each contain $K$. Therefore, Lemma \ref{lem:G1-H1} gives that $G_1 = KH_1$ for $H_1 := G_1 \cap H$ and that $G_2 = K H_2$ for $H_2 := G_2 \cap H$. Now, we claim that
\begin{equation} \label{eq:long-inclusion}
 KZ_H(H_1) \subseteq Z_G(G_1) \subseteq Z_G(H_1) \subseteq K Z_H(H_1).   
\end{equation}
For the first inclusion, let $kh \in KZ_H(H_1)$, and let $k'h' \in KH_1 = G_1$. Then
$$(kh)(k'h')(kh)^{-1} = k' (hh'h^{-1}) = h',$$
so $kh \in Z_G(G_1)$. The second inclusion follows immediately from $H_1 \subseteq G_1$. For the final inclusion, let $g \in Z_G(H_1)$ and write $g = kh$ with $k \in K$ and $h \in H$. Then for any $h_1 \in H_1$,
$$(kh)h_1(kh)^{-1} = h_1 \implies hh_1h^{-1} = h_1,$$
so $h \in Z_H(H_1)$ and $g \in KZ_H(H_1)$. It follows that all of the groups in (\ref{eq:long-inclusion}) are equal. The same result holds for $G_2$ and $H_2$. It follows that
$$Z_H(H_1) = Z_G(H_1) \cap H = KZ_H(H_1) \cap H = G_2 \cap H = H_2.$$
Similarly, $Z_H(H_2) = H_1$. This completes the proof.
\end{proof}

\begin{lemma} \label{lemma:lift-subgroup}
Suppose that $G$ equals a product of subgroups $G = KH$, where $K$ is central. If $(H_1, H_2)$ is a dual pair in $H$, then 
\begin{enumerate}[label=(\roman*)]
    \item $(KH_1, KH_2)$ is a dual pair in $G$, and
    \item $KH_1 \cap H = H_1$ and $KH_2 \cap H = H_2$.
\end{enumerate}
\end{lemma}

\begin{proof}
We would like to show that $Z_G(KH_1) = KH_2$ and $Z_G(KH_2)=KH_1$. Certainly, we have the inclusions $KH_2 \subseteq Z_G(KH_1)$ and $KH_1 \subseteq Z_G(KH_2)$. We now show that $Z_G(KH_1) \subseteq KH_2$. Suppose $t \in Z_G(H_1)$. Then since $G = KH$, we can write $t = kh$ with $k \in K$ and $h \in H$. Then for any $h' \in H_1$,
$$h'=(kh)h'(kh)^{-1} \implies h'=k^{-1}h'k = hh'h^{-1} \implies h \in H_2.$$
It follows that $t \in KH_2$, and hence that $Z_G(KH_1) \subseteq KH_2$. By a similar argument, we get that $Z_G(KH_1) \subseteq KH_1$. Therefore, $(KH_1, KH_2)$ is a dual pair in $G$, as desired.

We now prove that $KH_1 \cap H = H_1$. Certainly, $H_1 \subseteq KH_1 \cap H$. On the other hand, suppose that $kh \in KH_1 \cap H$. Then since $(KH_1, KH_2)$ is a dual pair, we have that
$$(kh)h_2(kh)^{-1} = h_2$$
for any $h_2 \in H_2$. It follows that $kh \in Z_H(H_2) = H_1$, and hence that $KH_1 \cap H = H_1$. The same argument gives that $KH_2 \cap H = H_2$, completing the proof.
\end{proof}

\begin{thm} \label{thm:subgroup-bijection}
Suppose that $G$ equals a product of subgroups $G = KH$, where $K$ is central. Then there exists a bijection
\begin{align*}
    \{ \text{dual pairs of }H \} & \leftrightarrow \{ \text{dual pairs of }G \} \\
    (H_1,H_2) & \leftrightarrow (KH_1,KH_2),
\end{align*}
where $\rightarrow$ is given by multiplication by $K$ and $\leftarrow$ is given by intersection with $H$.
\end{thm}

\begin{proof}
This follows immediately by putting together the above results.
\end{proof}

\section{Dual pairs in quotients} \label{sec:quotients}

Let $G$ be a complex reductive algebraic group. Having discussed the relationship between dual pairs in $G$ and dual pairs in certain subgroups of $G$, we now turn to the relationship between dual pairs in $G$ and dual pairs in certain quotients of $G$. For a subgroup $H$ of $G$ and a normal subgroup $N$ of $G$, let $\pi \colon G \rightarrow G/N$ denote the canonical projection, and define
$$K_{H,N}:= \{ tht^{-1}h^{-1} \; : \; t \in \pi^{-1}(Z_{G/N}(\pi (H))), \; h \in H \}.$$

\begin{lemma} \label{lem:quotient-centralizer}
Let $H$ be a connected subgroup of $G$, and let $N$ be a normal subgroup of $G$. If $K_{H,N}$ is discrete, then $Z_{G/N}(\pi (H)) = \pi (Z_G(H))$. 
\end{lemma}

\begin{proof}
The inclusion $Z_{G/N}(\pi (H)) \supseteq \pi (Z_G(H))$ is clear. For inclusion the other way, let $t \in \pi^{-1} (Z_{G/N}(\pi (H)))$. We would like to show that $t \in Z_G(H)$. By choice of $t$, we have that for any $h \in H$,
$$tht^{-1} = n_t(h) h \; \text{ for some } n_t(h) \in K_{H,N}.$$
Since multiplication and inversion of elements are continuous operations in an algebraic group, we see that $n_t$ defines a continuous function
\begin{align*}
    n_t\colon H & \rightarrow K_{H,N} \\
    h& \mapsto tht^{-1}h^{-1}.
\end{align*}
Now, since $n_t$ defines a continuous map from a connected group to a discrete group, the image of $n_t$ must be a single point in $K_{H,N}$. Since $n_t(1) = 1$, it follows that $n_t$ is trivial, and hence that $t \in Z_G(H)$, completing the proof.
\end{proof}

\begin{remark} \label{rmk:N-discrete}
Let $H$, $N$, and $K_{H,N}$ be as in Lemma \ref{lem:quotient-centralizer}. Note that $K_{H,N} \subseteq N$. Therefore, a weaker version of Lemma \ref{lem:quotient-centralizer} is that $Z_{G/N}(\pi (H)) = \pi (Z_G(H))$ whenever $N$ is discrete. 
\end{remark}

\begin{cor} \label{cor:SO-centralizer-general}
Let $H$ be a subgroup of $G$ with identity component $H^{\circ}$, and let $N$ be a normal subgroup of $G$ such that $K_{H,N}$ is discrete. Then
$$\pi^{-1}(Z_{G/N} (\pi (H)) ) \subseteq Z_{G}(H^{\circ}).$$
\end{cor}

\begin{proof}
This follows from the inclusions
$$\pi^{-1}(Z_{G/N} (\pi (H)) ) \subseteq \pi^{-1} (Z_{G/N} (\pi (H^{\circ}))) \subseteq Z_{G}(H^{\circ}),$$
where the second inclusion comes from Lemma \ref{lem:quotient-centralizer}.
\end{proof}

\begin{prop} \label{prop:connected-dp-descend} 
Let $(G_1, G_2)$ be a dual pair in $G$. Let $N$ be a normal subgroup of $G$ such that $K_{G_1,N}$ and $K_{G_2,N}$ are discrete. If $Z_G(G_i) = Z_G(G_i^{\circ})$ for $i=1,2$, then $(\pi(G_1), \pi (G_2))$ is a dual pair in $G/N$.
\end{prop}

\begin{proof}
By Corollary \ref{cor:SO-centralizer-general}, 
$$ Z_{G/N} ( \pi(G_1) ) \subseteq \pi(Z_G( G_1^{\circ} )) = \pi(Z_G(G_1)).$$
On the other hand, if $x \in Z_G(G_1)$, then 
$$\pi(x) \pi(g_1) \pi(x)^{-1} = \pi(xg_1x^{-1}) = \pi(g_1) \text{ for all } g_1 \in G_1,$$
meaning $\pi(x) \in Z_{G/N}(\pi(G_1))$. It follows that 
$$Z_{G/N} (\pi(G_1)) = \pi(Z_G(G_1)) = \pi(G_2).$$
Similarly, $Z_{G/N}(\pi(G_2)) = \pi(G_1)$, completing the proof.
\end{proof}

Note that if $G$ is a subgroup of $GL(n,\C)$ and if the \textit{standard representations} $G_i \hookrightarrow G$ and $G_i^{\circ} \hookrightarrow G$ are both irreducible, then $Z_G(G_i) = Z_G(G_i^{\circ})$ by Schur's lemma. Since we will be using this frequently throughout the paper, we state here a result about the irreducibility of the standard representations of the classical subgroups of $GL(n,\C)$:  
\begin{lemma}[Goodman and Wallach {\cite[Section 5.5]{Symmetry}}] \label{lem:standard-irreps}
The standard representations of the following groups are irreducible:
\begin{center}
\begin{tabular}{c c c}
$GL(n,\C)$ & \hspace{.5cm} & $(n \geq 1)$ \\
$SL(n,\C)$ & \hspace{.5cm} & $(n \geq 1)$ \\
$Sp(2n,\C)$ & \hspace{.5cm} & $(n \geq 1)$ \\
$O(n,\C)$ & \hspace{.5cm} & $(n \geq 1)$ \\
$SO(n,\C)$ & \hspace{.5cm} & $(n \neq 2)$ \\
\end{tabular}
\end{center}
\end{lemma}

\begin{prop} \label{prop:quotient-dp-lifts}
Let $G$ be a complex reductive algebraic group, and let $N$ be a normal subgroup of $G$. Let $(\overline{H}_1, \overline{H}_2)$ be a connected dual pair in $G/N$, and set $H_i := \pi^{-1} (\overline{H}_i)$. If $K:= \{ h_1 h_2 h_1^{-1} h_2^{-1} \; : \; h_i \in H_i \}$ is discrete, then $(H_1, H_2 )$ is a dual pair in $G$.
\end{prop}

\begin{proof}
To start, we note that $Z_G(H_1) \subseteq H_2$. To this end, suppose that $g \in Z_G(H_1)$. Then $g h_1 g^{-1} = h_1$ for all $h_1 \in H_1$, so $\pi(g) \pi(h_1) \pi(g)^{-1} = \pi(h_1)$ for all $h_1 \in H_1$. It follows that $\pi(g) \in \overline{H}_2$, and hence that $g \in H_2$. Therefore, $Z_G(H_1) \subseteq H_2$. By the same reasoning, $Z_G(H_2) \subseteq H_1$.  

Next, note that $H_i$ is an $N$-bundle over $G/N$. Therefore, since $\overline{H}_1$ and $\overline{H}_2$ are connected, we get that
$$H_i / H_i^{\circ} = \overline{H}_i / \overline{H}_i^{\circ} = \{ 1 \}.$$ 
In other words, $H_1$ and $H_2$ are connected. Consider the map
\begin{align*}
\mu \colon H_1 \times H_2 & \rightarrow K \\
(h_1, h_2) & \mapsto h_1 h_2 h_1^{-1} h_2^{-1}.
\end{align*}
Since $H_1 \times H_2$ is connected, we see that $\mu$ is a continuous map from a connected group to a discrete group, and hence must be constant. Since $(1,1) \in H_1 \times H_2$ maps to 1 under $\mu$, it follows that $\mu$ is trivial, and hence that $(H_1, H_2)$ is a dual pair in $G$. 
\end{proof}

\section{Dual pairs in Lie algebras} \label{sec:dps-in-lie-algs}

Let $G$ be a complex reductive Lie group. In this section, we discuss the dual pairs in $\mf{g} := \Lie (G)$ and their relationship to the dual pairs in $G$. To start, the following proposition gives a condition under which a Lie group dual pair gives rise to a corresponding Lie algebra dual pair.

\begin{prop} \label{prop:1}
Let $(G_1, G_2)$ be a dual pair in $G$, and set $\mf{g}_i := \Lie (G_i)$. If $Z_G(\langle \exp \mf{g}_i \rangle) = Z_G(G_i)$ for $i\in \{1,2\}$, then $(\mf{g}_1,\mf{g}_2)$ is a dual pair in $\mf{g}$.
\end{prop}

\begin{proof}
By \cite[Theorem 5.1]{sepanski}, for $X,Y \in \mf{g}$, $[X,Y]=0$ if and only if $e^{tX}$ and $e^{sY}$ commute for $s,t \in \mathbb{R}$. Therefore, it is clear that $\mf{g}_1 \subseteq \mf{z}_{\mf{g}}(\mf{g}_2)$ and $\mf{g}_2 \subseteq \mf{z}_{\mf{g}} (\mf{g}_1)$. Next, suppose that $[X,Y] = 0$ for all $Y \in \mf{g}_2$. Then $\exp (X) \in Z_G( \langle \exp \mf{g}_2 \rangle ) = Z_G(G_2) = G_1$, which gives that $X \in \mf{g}_1$. Therefore, $\mf{z}_{\mf{g}} (\mf{g}_2) \subseteq \mf{g}_1$. It follows that $\mf{g}_1 = \mf{z}_{\mf{g}}(\mf{g}_2)$, and by the same argument, that $\mf{g}_2 = \mf{z}_{\mf{g}} (\mf{g}_1)$.
\end{proof}

\begin{remark} \label{rmk:connected}
Note that if $G_1$ and $G_2$ are connected, then $\exp \mf{g}_i$ generates $G_i$ by \cite[Theorem 4.6]{sepanski}. Therefore, a particular consequence of Proposition \ref{prop:1} is that $(\mf{g}_1, \mf{g}_2)$ is a dual pair in $\mf{g}$ whenever $(G_1, G_2)$ is a connected dual pair in $G$. 
\end{remark}

Conversely, for any dual pair in a complex reductive Lie algebra $\mf{g}$, the following proposition shows that there is a corresponding Lie group dual pair in any complex reductive Lie group $G$ with $\Lie (G) = \mf{g}$.

\begin{prop} \label{prop:2}
Let $(\mf{g}_1, \mf{g}_2)$ be a dual pair in $\mf{g}$. Set $G_2 = Z_G(\langle \exp \mf{g}_1 \rangle)$ and $G_1 = Z_G(G_2)$. Then $(G_1, G_2)$ is a reductive dual pair in $G$ with Lie algebras $(\mf{g}_1, \mf{g}_2)$. 
\end{prop}

\begin{proof}
We first show that $(G_1, G_2)$ is a dual pair in $G$, and we then show that $(\mf{g}_1, \mf{g}_2)$ are the corresponding Lie algebras. To start, note that
$$Z_G(G_1) = Z_G(Z_G(G_2)) = Z_G(Z_G(Z_G(\langle \exp \mf{g}_1 \rangle))) = Z_G(\langle \exp \mf{g}_1 \rangle) = G_2.$$
Since $Z_G(G_2) = G_1$ by definition of $G_1$, we see that $(G_1, G_2)$ is indeed a dual pair in $G$.

Next, note that 
$$\Lie (G_2) = \Lie (Z_G(G_1^0)) = \mf{z}_{\mf{g}} (G_1^0) = \mf{z}_{\mf{g}} (\mf{g}_1) = \mf{g}_2,$$
where we have used \cite[Exercise 4.22]{sepanski}. Finally, we claim that $\Lie (G_1) = \mf{g}_1$. To see this, we first note that since $\langle \exp \mf{g}_1 \rangle \subseteq Z_G(Z_G(\langle \exp \mf{g}_1 \rangle)) = G_1$, we have that $\mf{g}_1 = \Lie (\langle \exp \mf{g}_1 \rangle) \subseteq \Lie (G_1)$. On the other hand, note that
$$\Lie (G_1) = \Lie (Z_G(G_2)) = \mf{z}_{\mf{g}} ( G_2 ).$$ 
So if $X \in \Lie (G_1)$, then $e^{tY} X e^{-tY} = X$ for all $Y \in \mf{g}_2$, meaning $[Y,X] = 0$ for all $Y \in \mf{g}_2$. It follows that $X \in \mf{z}_{\mf{g}}(\mf{g}_2) = \mf{g}_1$, and hence that $\Lie (G_1) = \mf{g}_1$, completing the proof. 
\end{proof}

Therefore, with an understanding of Lie algebra dual pairs, one can understand (some of the) dual pairs in the corresponding Lie groups. Since Rubenthaler outlines a classification of reductive dual pairs in reductive Lie algebras in \cite{rubenthaler}, one might naturally suggest trying to use Rubenthaler's classification to understand Lie group dual pairs. However, Rubenthaler's classification is rather complicated, and using his work to obtain complete lists of dual pairs requires a lot of nontrivial work from the reader. Fortunately, as we will see later in this paper, the dual pairs in the complex classical Lie groups can be completely classified using elementary methods (see Subsections \ref{subsec:GL}, \ref{subsec:SL}, \ref{subsec:Sp}, \ref{subsec:O}, \ref{subsec:SO}). Using these classifications in conjunction with Propositions \ref{prop:1} and \ref{prop:2}, one can moreover obtain complete classifications of the dual pairs in the complex classical Lie algebras (see Subsections \ref{subsec:sl-LA}, \ref{subsec:gl-LA}, \ref{subsec:sp-LA}, \ref{subsec:so-LA}).

\section{Embeddings} \label{sec:Embeddings}

In this section, we describe various embeddings of pairs of complex classical groups into larger complex classical groups. These embeddings will be very relevant for the classifications of dual pairs in the remainder of the paper.

\subsection{Embeddings into \texorpdfstring{$GL(X \otimes Y)$}{GL(n,C)}} \label{subsec:GL(XxY)}

Let $X$ and $Y$ be finite-dimensional complex vector spaces of dimensions $k$ and $\ell$, respectively. If we choose a basis of $X$, then $\End(X \otimes Y)$ gets identified with $k \times k$ matrices with entries in $\End (Y)$. We define the embedding $\iota \colon GL(X) \rightarrow GL(X \otimes Y)$ by
$$ x \xmapsto{\iota} \begin{pmatrix} x_{11}I_\ell & \cdots & x_{1k} I_\ell \\ \vdots & \ddots & \vdots \\ x_{k1}I_\ell & \cdots & x_{kk}I_\ell \end{pmatrix} \in GL(X \otimes Y),$$
and the embedding $\kappa \colon GL(Y) \rightarrow GL(X \otimes Y)$ by
$$y \xmapsto{\kappa} \begin{pmatrix} y & & \\ & \ddots & \\ & & y \end{pmatrix} \in GL(X \otimes Y),$$
where the image of $y$ under $\kappa$ has $k$ copies of $y$ along its diagonal. As we will soon see, $(\iota(GL(X)), \kappa(GL(Y)))$ is an example of a dual pair in $GL ( X\otimes Y)$.

\subsection{Embeddings into \texorpdfstring{$Sp(X \otimes Y)$}{Sp(2n,C)}}

Let $X$ and $Y$ be finite-dimensional complex vector spaces of dimensions $k$ and $2\ell$, respectively. Additionally, assume that $X$ admits an orthogonal form and that $Y$ admits a symplectic form (so that $X \otimes Y$ admits a symplectic form). Write $I_k$ for the orthogonal form on $X$ and 
$$\Omega_{2\ell} := \begin{pmatrix} & I_{\ell} \\ -I_{\ell} & \end{pmatrix}$$
for the symplectic form on $Y$. As before, if we choose a basis of $X$, then $\End(X \otimes Y)$ gets identified with $k \times k$ matrices with entries in $\End (Y)$. Then the forms on $X$ and $Y$ induce the symplectic form 
$$ \Omega := \begin{pmatrix} \Omega_{2\ell} & & \\ & \ddots & \\ & & \Omega_{2\ell} \end{pmatrix}$$
on $X \otimes Y$ so that $Sp(X \otimes Y) = \{ g \in GL(X \otimes Y) \; : \; g^t \Omega g = \Omega \}$. We get an embedding $\iota \colon O(X) \hookrightarrow Sp(X \otimes Y)$, where $\iota$ is as in Subsection \ref{subsec:GL(XxY)} (with $\ell$ replaced by $2\ell$). One can easily check that $\iota(x)^t \Omega \iota(x) = \Omega$ by recalling that $x^tx = I_k$ for all $x \in O(X)$. Similarly, we get an embedding $\kappa \colon Sp(Y) \rightarrow Sp(X \otimes Y)$, where $\kappa$ is as in Subsection \ref{subsec:GL(XxY)} (with $k$ replaced by $2k$). One can easily check that $\kappa (y)^t \Omega \kappa (y) = \Omega$ for $y \in Sp(Y)$ by recalling that $y^t \Omega_{2\ell} y = \Omega_{2\ell}$ for all $y \in Sp(Y)$. We will soon see that $(\iota (O(X)), \kappa (Sp(Y)))$ is an example of a dual pair in $Sp(X \otimes Y)$. (In fact, this is an example of a ``classical"/``irreducible" dual pair as studied in \cite{HoweRemarks,schmidt}.)

\subsection{Embeddings into \texorpdfstring{$Sp(X \otimes Y \oplus X^* \otimes Y^*)$}{Sp(2n,C)}} \label{subsec:GL-Sp-embedding}

Let $X$ and $Y$ be finite-dimensional complex vector spaces of dimensions $k$ and $\ell$, respectively. We define embeddings $GL(X) \hookrightarrow Sp(X \otimes Y \oplus X^* \otimes Y^*)$ and $GL(Y) \hookrightarrow Sp(X \otimes Y \oplus X^* \otimes Y^*)$ as follows:
\begin{center}
    \begin{tabular}{l l c l}
        $x$ & $\xmapsto{\iota}$ & $\begin{pmatrix} a_{11}I_\ell & \cdots & a_{1k}I_\ell \\ \vdots & \ddots & \vdots \\ a_{k1}I_\ell & \cdots & a_{kk}I_\ell \end{pmatrix}$ & $\xmapsto{\tau} \begin{pmatrix} \iota(x) & \\ & (\iota(x)^t)^{-1} \end{pmatrix}$ ,  \\
        $y$ & $\xmapsto{\kappa}$ & $\begin{pmatrix} y & & \\ & \ddots & \\ & & y \end{pmatrix}$ & $\xmapsto{\tau} \begin{pmatrix} \kappa(y) & \\ & (\kappa(y)^t)^{-1} \end{pmatrix}$.
    \end{tabular}
\end{center}
Here, $\iota$ and $\kappa$ are as defined in Subsection \ref{subsec:GL(XxY)}. It is straightforward to check that both $\tau (\iota (x))$ and $\tau (\kappa(y))$ preserve the symplectic form
$$\Omega_{2k\ell} = \begin{pmatrix} & I_{k\ell} \\ -I_{k\ell} & \end{pmatrix}$$
on $X \otimes Y \oplus X^* \otimes Y^*$. We will soon see that $(\tau(\iota (GL(X))), \tau(\kappa (GL(Y))))$ is an example of a dual pair in $Sp(X \otimes Y \oplus X^* \otimes Y^*)$. (In fact, this is another example of a ``classical"/``irreducible" dual pair as studied in \cite{HoweRemarks,schmidt}.)

\subsection{Embeddings into \texorpdfstring{$O(X\otimes Y)$}{O(n,C)}} 

Let $X$ and $Y$ be finite-dimensional complex vector spaces of dimensions $k$ and $\ell$, respectively. Additionally, assume that $X$ and $Y$ admit orthogonal forms $I_k$ and $I_\ell$. Then the forms on $X$ and $Y$ induce the orthogonal form $I_{k\ell}$ on $X\otimes Y$. Then we get embeddings $\iota \colon O(X) \hookrightarrow O(X \otimes Y)$ and $\kappa \colon O(Y) \hookrightarrow O(X \otimes Y)$, where $\iota$ and $\kappa$ are as in Subsection \ref{subsec:GL(XxY)}.

Next, let $X$ and $Y$ be finite-dimensional complex vector spaces of dimensions $2k$ and $2\ell$, respectively. Additionally, assume that $X$ and $Y$ admit symplectic forms $\Omega_{2k}$ and $\Omega_{2\ell}$. Then the forms on $X$ and $Y$ induce the orthogonal form
$$\Omega := \begin{pmatrix} & & & \Omega_{2\ell} & & \\ & & & & \ddots & \\ & & & & & \Omega_{2\ell} \\ -\Omega_{2\ell} & & & & & \\ & \ddots & & & & \\ & & -\Omega_{2\ell} & & & \\ \end{pmatrix}$$
on $X\otimes Y$ so that $O(X\otimes Y) = \{ g\in GL(X\otimes Y) \; : \; g^t \Omega g = \Omega \}$. We get embeddings $\iota \colon Sp(X) \hookrightarrow O(X\otimes W)$ and $\kappa \colon Sp(Y) \hookrightarrow O(X\otimes Y)$ (with $k$ and $\ell$ replaced by $2k$ and $2\ell$).

We will soon see that $( \iota (O(X)), \kappa (O(Y)))$ and $( \iota (Sp(X)), \kappa (Sp(Y)))$ are examples of dual pairs in $O(X\otimes Y)$. (These are again examples of ``classical"/``irreducible" dual pairs as studied in \cite{HoweRemarks,schmidt}.)

\subsection{Embeddings into \texorpdfstring{$O(X \otimes Y \oplus X^* \otimes Y^*)$}{O(n,C)}}

Let $X$ and $Y$ be finite-dimensional complex vector spaces of dimensions $k$ and $\ell$, respectively. Then we get embeddings $\tau \circ \iota \colon GL(X) \hookrightarrow O(X \otimes Y \oplus X^* \otimes Y^*)$ and $\tau \circ \kappa \colon GL(Y) \hookrightarrow O(X \otimes Y \oplus X^* \otimes Y^*)$ in the same way as in Subsection \ref{subsec:GL-Sp-embedding}, where the orthogonal form on $X \otimes Y \oplus X^* \otimes Y^*$ is 
$$\begin{pmatrix} & I_{k\ell} \\ I_{k\ell} & \end{pmatrix}.$$
These embeddings again yield examples of ``classical"/``irreducible" dual pairs \cite{HoweRemarks,schmidt}.

\section{Dual pairs in \texorpdfstring{$GL(n,\mathbb{C})$}{GL(n,C)}, \texorpdfstring{$SL(n,\mathbb{C})$}{SL(n,C)}, \texorpdfstring{$PGL(n,\mathbb{C})$}{PGL(n,C)}, and their Lie algebras} \label{sec:gl}

\subsection{Dual pairs in \texorpdfstring{$GL(n,\mathbb{C})$}{GL(n,C)}} \label{subsec:GL}

Let $U$ be a finite-dimensional complex vector space, and let $H$ be a complex reductive algebraic group. Since we are working in characteristic zero, every algebraic representation of $H$ is completely reducible. Therefore, for any algebraic representation $\varphi \colon H \rightarrow GL(U)$, Schur's lemma gives the decomposition
$$U \simeq \bigoplus_{i} V_i \otimes \Hom_H (V_i,U),$$
where the $V_i$'s are the nonisomorphic irreducible subrepresentations of $U$. This decomposition will be crucial for our classification of dual pairs in $GL(U)$.

\begin{prop} \label{prop:dual-pairs-GL}
Let $U$ be a finite-dimensional complex vector space. Then the dual pairs of $GL(U)$ are exactly the pairs of groups of the form
$$\left ( \prod_{i=1}^{r} GL(V_i), \; \prod_{i=1}^{r} GL(\Hom_H (V_i, U) ) \right ),$$
where $H$ is a complex reductive algebraic group with algebraic representation $\varphi \colon H \rightarrow GL(U)$, and where $(\varphi_1, V_1), \ldots, (\varphi_r, V_r)$ are the nonisomorphic irreducible subrepresentations of $(\varphi, U)$. 
\end{prop}

\begin{proof}
Set $W_i:= \Hom_H(V_i, U)$.  \\

\noindent \textit{Step 1: $Z_{GL(U)} \left ( \prod_{i=1}^{r} GL(V_i) \right ) = \prod_{i=1}^{r} GL(W_i) = GL(U)^{\varphi}$.} \\

Let $t \in Z_{GL(U)} \left ( \prod_{i=1}^{r} GL(V_i) \right )$. Note that $t$ commutes with any element of $\varphi (H)$ (since $\varphi (H) \subseteq \prod_{i=1}^{r} GL(V_i)$). In other words, $t$ is $H$-linear. Since $V_1, \ldots, V_r$ are nonisomorphic irreducible representations, Schur's lemma gives that there are no nontrivial $H$-linear maps between them. Therefore, $t$ necessarily preserves each $V_i \otimes W_i$, and hence can be decomposed as $t = t_1 \oplus \cdots \oplus t_r$, where $t_i \colon V_i \otimes W_i \rightarrow V_i \otimes W_i$ is $H$-linear. Applying Schur's lemma again gives that the action of $t_i$ on each $V_i$ is given by $\lambda \cdot I_{V_i}$ for some $\lambda \in \mathbb{C}^{\times}$. It follows that $t_i \in GL(W_i)$, giving that $t \in \prod_{i=1}^{r} GL(W_i)$ and hence that $Z_{GL(U)} \left ( \prod_{i=1}^{r} GL(V_i) \right ) \subseteq \prod_{i=1}^{r} GL(W_i)$. On the other hand, it is clear that $\prod_{i=1}^{r} GL(W_i) \subseteq Z_{GL(U)} \left ( \prod_{i=1}^{r} GL(V_i) \right )$. This same argument also shows that $\prod_{i=1}^{r} GL(W_i)$ consists exactly of the $H$-linear elements of $GL(U)$, completing Step 1. \\

\noindent \textit{Step 2: Each $W_i$ is an irreducible representation of $\prod_{i=1}^{r} GL(W_i)$ with multiplicity space $V_i$.} \\

Set $H':= \prod_{i=1}^{r} GL(W_i)$ and consider the representation $\rho \colon H' \hookrightarrow GL(U)$. We will show that the $W_i$'s are precisely the irreducible subrepresentations of $\rho$, and that $W_i$ has multiplicity space $V_i$. Certainly, each $W_i$ is a subrepresentation of $U$, since $\rho (h)w \in W_i$ for each $h \in H'$ and $w \in W_i$. Moreover, each $W_i$ is irreducible by Lemma \ref{lem:standard-irreps}. Therefore, we now have two decompositions of $U$: one as an $H$-representation and one as an $H'$-representation. Combining these gives that
$$U \simeq \bigoplus V_i \otimes \Hom_H(V_i, U) \simeq \bigoplus W_i \otimes \Hom_{H'}(W_i,U).$$
Since $W_i = \Hom_H (V_i, U)$, this shows that $V_i = \Hom_{H'}(W_i, U)$, completing Step 2. \\

\noindent \textit{Step 3: $Z_{GL(U)}\left ( \prod_{i=1}^{r} GL(W_i) \right ) = \prod_{i=1}^{r} GL(V_i)$.} \\

Step 2 shows that we can repeat Step 1 with $H'$ in place of $H$ and with the roles of $V_i$ and $W_i$ to get that $Z_{GL(U)} (\prod_{i=1}^{r} GL(W_i)) = \prod_{i=1}^{r} GL(V_i)$. \\

\noindent \textit{Step 4: All dual pairs of $GL(U)$ are of this form.} \\

It was shown in Step 1 that $\prod_{i=1}^{r} GL(W_i) = GL(U)^{\varphi}$. Therefore, Step 4 follows from Remark \ref{rmk:all pairs}.
\end{proof}

\begin{thm} \label{thm:dual-pairs-in-GL}
The dual pairs of $GL(n,\C)$ are exactly the pairs of groups of the form
\begin{equation*} 
\left ( \prod_{i=1}^{r} GL(k_i) , \; \prod_{i=1}^r GL(\ell_i) \right ),
\end{equation*}
where $k_i, \ell_i > 0$, and where $n=\sum_{i=1}^{r} k_i \ell_i$.
\end{thm}

\begin{proof}
This follows from the proof of Proposition \ref{prop:dual-pairs-GL} by taking 
$$H = \prod_{i=1}^{r} GL(V_i)$$ 
(with $\dim V_i = k_i$) and applying Lemma \ref{lem:standard-irreps}.
\end{proof}

\begin{cor} \label{cor:connected}
Let $(G_1, G_2)$ be a dual pair in $GL(n,\C)$. Then $G_1$ and $G_2$ are connected.
\end{cor}

\subsection{Dual pairs in \texorpdfstring{$SL(n,\mathbb{C})$}{SL(n,C)}} \label{subsec:SL}

\begin{thm} \label{thm:dual-pairs-in-SL}
The dual pairs of $SL(n,\C)$ are exactly the pairs of groups of the form
$$\left ( \Big ( \prod_{i=1}^{r} SL(k_i,\C) \Big ) \times \mathbb{C}^{r-1} , \; \Big ( \prod_{i=1}^{r} SL(\ell_i,\C) \Big ) \times \mathbb{C}^{r-1} \right ),$$
where $k_i, \ell_i > 0$, and where $n = \sum_{i=1}^{r} k_i \ell_i$.
\end{thm}

\begin{proof}
To start, we show that $GL(n,\C) = \C^{\times} \cdot SL(n,\C)$. For any $g \in GL(n,\C)$, let $d_g := \text{diag}((1/\det g)^{1/n}, \ldots, (1/\det g)^{1/n})$. Then $\det (g d_g) = \det (g) \det (d_g) = 1$. Moreover, $d_g \in Z(GL(n,\C)) = \C^{\times}$. Defining $g':=g d_g$, we can write $g = d_g^{-1}g'$, where $g' \in SL(n,\C)$ and $d_g^{-1} \in \C^{\times}$. In this way, we see that $GL(n,\C) = \C^{\times} \cdot SL(n,\C)$, as desired.

With this established, Theorems \ref{thm:subgroup-bijection} and \ref{thm:dual-pairs-in-GL} give that the dual pairs in $SL(n,\C)$ are exactly the pairs of groups of the form
$$\left ( \Big ( \prod_{i=1}^{r} GL(k_i) \Big ) \cap SL(n,\C) , \; \Big ( \prod_{i=1}^{r} GL(\ell_i) \Big ) \cap SL(n,\C)  \right ),$$
where $n=\sum_{i=1}^{r} k_i \ell_i$. The proof is completed by the observation that
$$\Big ( \prod_{i=1}^{r} GL(k_i) \Big ) \cap SL(n,\C) = \Big ( \prod_{i=1}^{r} SL(k_i) \Big ) \times \mathbb{C}^{r-1}$$
(and similarly for the other member of the pair).
\end{proof}

\subsection{Connected dual pairs in \texorpdfstring{$PGL(n,\mathbb{C})$}{PGL(n,C)}} \label{subsec:PGL}

In this subsection, we classify the connected dual pairs in $PGL(n,\C)$. Let $p \colon GL(n,\C) \rightarrow PGL(n,\C)$ be the canonical projection.

\begin{lemma} \label{lem:roots-of-unity}
Suppose $xyx^{-1} = cy$, where $x,y \in GL(n,\mathbb{C})$ and $c \in Z(GL(n,\mathbb{C})) \simeq \mathbb{C}^{\times}$. Then $c$ is an $n$-th root of unity. 
\end{lemma}

\begin{proof}
Using that the determinant is multiplicative, we see that $\det (xyx^{-1}) = \det (y)$. Therefore, the relation $xyx^{-1} = cy$ gives that 
$$\det (y) = \det (cy) = c^n \det (y).$$
This shows that $c^n = 1$, or that $c$ is an $n$-th root of unity (not necessarily primitive).
\end{proof}

\begin{prop} \label{prop:GL-PGL-descent}
Let $(G_1, G_2)$ be a dual pair in $GL(n,\C)$. Then $(p (G_1), p(G_2))$ is a dual pair in $PGL(n,\C)$.
\end{prop}

\begin{proof}
By Corollary \ref{cor:connected}, $G_1$ and $G_2$ are connected. Write $Z:= Z(GL(n,\C)) \simeq \mathbb{C}^{\times}$. By Lemma \ref{lem:roots-of-unity}, both $K_{G_1,Z}:= \{ tg_1t^{-1}g_1^{-1} \; : \; t \in p^{-1}(Z_{PGL(n,\C)}(p (G_1))), \; g_1 \in G_1 \}$ and $K_{G_2,Z}:= \{ tg_2t^{-1}g_2^{-1} \; : \; t \in p^{-1}(Z_{PGL(n,\C)}(p (G_2))), \; g_2 \in G_2 \}$ contain only $n$-th roots of unity, and hence are discrete. Therefore, Proposition \ref{prop:connected-dp-descend} applies, completing the proof.
\end{proof}

Proposition \ref{prop:GL-PGL-descent} shows that every dual pair in $GL(n,\C)$ descends to a dual pair in $PGL(n,\C)$ under the canonical projection. In the other direction, the following proposition shows that the connected dual pairs in $PGL(n,\C)$ lift to dual pairs in $GL(n,\C)$.

\begin{prop} \label{prop:connected-connected}
Let $(\overline{G}_1, \overline{G}_2)$ be a connected dual pair in $PGL(n,\C)$, and define $G_1:=p^{-1}(\overline{G}_1)$ and $G_2:= p^{-1}(\overline{G}_2)$. Then $(G_1, G_2)$ is a dual pair in $GL(n,\C)$.
\end{prop}

\begin{proof}
This follows immediately from Proposition \ref{prop:quotient-dp-lifts} and Lemma \ref{lem:roots-of-unity}.
\end{proof}

Classifying the disconnected dual pairs in $PGL(n,\C)$ is a much more difficult question. In a forthcoming paper, we show that the problem of classifying dual pairs in $PGL(n,\C)$ can be reduced to the problem of classifying certain pairs of extensions
$$1 \rightarrow G_1^{\circ} \rightarrow G_1 \rightarrow \Gamma \rightarrow 1 \hspace{.75cm} \text{ and } \hspace{.75cm} 1 \rightarrow G_2^{\circ} \rightarrow G_2 \rightarrow \hat{\Gamma} \rightarrow 1,$$
where $G_1^{\circ} = Z_{GL(n,\C)}(G_2)$ and $G_2^{\circ} = Z_{GL(n,\C)}(G_1)$ are members of some (possibly different) $GL(n,\C)$ dual pairs, and where $\Gamma$ is a finite abelian group. In this same forthcoming paper, we also classify such pairs of extensions using an in-depth analysis of the $\hat{\Gamma}$-orbits (resp.~$\Gamma$-orbits) of the $G_1$-irreducibles and $G_2^{\circ}$-irreducibles (resp.~$G_2$-irreducibles and $G_1^{\circ}$-irreducibles).

\subsection{Dual pairs in \texorpdfstring{$\mf{sl}(n,\C)$}{sl(n,C)}} \label{subsec:sl-LA}

\begin{thm} \label{thm:sln-dps}
The dual pairs in $\mf{sl}(n,\C)$ are exactly the pairs of subalgebras of the form
$$\left ( \Big ( \bigoplus_{1 = 1}^{r} \mf{sl}(k_i,\C) \Big ) \oplus \mathbb{C}^{r-1}, \; \Big ( \bigoplus_{i =1}^r \mf{sl}(k_i,\C) \Big ) \oplus \mathbb{C}^{r-1} \right ),$$
where $k_i, \ell_i > 0$, and where $\sum_i k_i \ell_i = n$.
\end{thm}

\begin{proof}
By Theorem \ref{thm:dual-pairs-in-SL}, 
$$\Bigg ( \Big ( \prod_{i=1}^r GL(k_i,\C) \Big ) \cap SL(n,\C), \; \Big ( \prod_{i=1}^r GL(\ell_i,\C) \Big ) \cap SL(n,\C)  \Bigg )$$
is a dual pair in $SL(n,\C)$. Moreover, recalling that $\exp \mf{gl}(n,\C) = GL(n,\C)$, it is not hard to check that 
$$\exp \left ( \Big ( \bigoplus_{i=1}^r \mf{gl}(k_i,\C)  \Big ) \cap \mf{sl}(n,\C) \right ) = \Big ( \prod_{i=1}^r GL(k_i,\C) \Big ) \cap SL(n,\C)  =  \Big ( \prod_{i=1}^{r} SL(k_i,\C) \Big ) \times \C^{r-1} .$$
Therefore, Proposition \ref{prop:1} gives that the pairs of subalgebras of the form described in the proposition statement are indeed dual pairs. 

On the other hand, suppose $(\mf{g}_1,\mf{g}_2)$ is a dual pair in $\mf{sl}(n,\C)$. Set $G_1^0 = \langle \exp \mf{g}_1 \rangle$, $G_2 = Z_G(G_1^0)$, and $G_1 = Z_G(G_2)$. Then by Proposition \ref{prop:2} and Theorem \ref{thm:dual-pairs-in-SL}, $(G_1, G_2)$ is of the form
$$(G_1, G_2) = \Bigg ( \Big ( \prod_{i=1}^{r} SL(k_i,\C) \Big ) \times \C^{r-1} , \; \Big ( \prod_{i=1}^{r} SL(\ell_i,\C) \Big ) \times \C^{r-1} \Bigg )  ,$$ 
where $n = \sum_{i=1}^{r} k_i \ell_i$. Moreover, Proposition \ref{prop:2} gives that 
$$\Lie (G_1) = \mf{g}_1 = \mf{sl}(k_1,\C) \oplus \cdots \oplus \mf{sl}(k_r,\C) \oplus \C^{r-1}$$
and that
$$\Lie (G_2) = \mf{g}_2 = \mf{sl}(\ell_1,\C) \oplus \cdots \oplus \mf{sl}(\ell_r,\C) \oplus \C^{r-1},$$
completing the proof.
\end{proof}

This result, obtained here from our analysis of Lie group dual pairs and the correspondence between Lie group and Lie algebra dual pairs, can also be obtained from a Lie algebras context alone. More precisely, we can make use of Rubenthaler's reduction of the classification of reductive dual pairs in reductive Lie algebras to the classification of semisimple dual pairs in simple Lie algebras \cite[Section 5]{rubenthaler}. This reduction requires the notion of a \textit{Levi subalgebra}: 

Let $\mf{g}$ be a semisimple Lie algebra, and let $\mf{h}$ be a Cartan subalgebra of $\mf{g}$. Let $\Psi = (\alpha_1 , \ldots , \alpha_n)$ be a set of simple roots associated to $(\mf{g}, \mf{h})$. For $\theta \subseteq \Psi$, let $H_{\theta}$ be the unique element of $\mf{h}$ determined by
  \[
    \left\{\begin{array}{ll}
        \alpha (H_{\theta}) = 0 & \text{if } \alpha \in \theta \\
        \alpha (H_{\theta}) = 2 & \text{if } \alpha \in \Psi \setminus \theta         \end{array}\right.
  \]  
For $p \in \mathbb{Z}$, set
$$d_p (\theta) = \{ X \in \mf{g} \; \vert \; [H_{\theta}, X] = 2pX \}, \hspace{.5cm} \text{ so that } \hspace{.5cm} \mf{g} = \bigoplus_{p \in \mathbb{Z}} d_p (\theta).$$

\begin{defn}[Rubenthaler {\cite[Section 1.2]{rubenthaler}}]
The reductive algebra $\ell_{\theta} := d_0 (\theta)$ is called the \textbf{standard Levi subalgebra associated to $\theta$}.
\end{defn}

\begin{lemma} \label{lem:sl-levis}
The Levi subalgebras of $\mf{sl}(n,\C)$ are exactly the subalgebras of the form 
$$\mf{sl}(n_1,\C) \oplus \cdots \oplus \mf{sl}(n_r,\C) \oplus \C^{r-1},$$
where $\sum_i n_i = n$. 
\end{lemma}

\begin{proof}
Let $\mf{h} := \{ \diag ( z_1 , \ldots , z_n ) \; \vert \; z_i \in \C , \, \sum_i z_i = 0 \}$. Then $\mf{h}$ is a Cartan subalgebra of $\mf{sl}(n,\C)$, and the set $\Psi := \{ \alpha_i := \varepsilon_i - \varepsilon_{i+1} \; \vert \; 1 \leq i \leq n-1 \}$ is a set of simple roots for the root system corresponding to $(\mf{sl}(n,\C), \mf{h})$. As $\theta$ varies over subsets of $\Psi$, we see that $H_{\theta}$ varies over diagonal matrices $\diag (z_1, \ldots , z_n) \in \mf{h}$ such that $z_i - z_{i+1} \in \{ 0,2 \}$ for all $1 \leq i \leq n-1$. In this way, we see that the Levi subalgebras of $\mf{sl}(n,\mathbb{C})$ are exactly the subalgebras
$$ \left [ \mf{gl}(n_1, \mathbb{C}) \oplus \cdots \oplus \mf{gl}(n_r,\mathbb{C}) \right ] \cap \mf{sl}(n,\mathbb{C}) = \mf{sl}(n_1, \mathbb{C}) \oplus \cdots \oplus \mf{sl}(n_r,\mathbb{C}) \oplus \mathbb{C}^{r-1},$$
completing the proof.
\end{proof}

\begin{proof}[Alternative proof of Theorem \ref{thm:sln-dps}]
To start, let us confirm that a pair of this form is in fact a dual pair in $\mf{sl}(n,\C)$. To this end, note first that by \cite[Lemma 5.1]{rubenthaler} and \cite[Lemma 5.3]{rubenthaler},
$$(\mf{sl}(k_1,\C)\oplus \cdots \oplus \mf{sl}(k_r,\C) \oplus \C^{r-1}, \; \mf{sl}(\ell_1,\C) \oplus \cdots \oplus \mf{sl}(\ell_r,\C) \oplus \C^{r-1}  )$$ 
is a dual pair in $\mf{sl}(k_1 \ell_1 ,\C) \oplus \cdots \oplus \mf{sl}( k_r \ell_r ,\C) \oplus \C^{r-1}$. Then since 
$$\mf{sl}(k_1 \ell_1 ,\C) \oplus \cdots \oplus \mf{sl}( k_r \ell_r ,\C) \oplus \C^{r-1}$$ 
is a Levi subalgebra of $\mf{sl}(n,C)$ (by Lemma \ref{lem:sl-levis}), \cite[Lemma 5.2]{rubenthaler} gives that 
$$(\mf{sl}(k_1,\C)\oplus \cdots \oplus \mf{sl}(k_r,\C) \oplus \C^{r-1}, \; \mf{sl}(\ell_1,\C) \oplus \cdots \oplus \mf{sl}(\ell_r,\C) \oplus \C^{r-1}  )$$ 
is a dual pair in $\mf{sl}(n,\C)$. 

Conversely, suppose we have a dual pair $(\mf{a}, \mf{b})$ in $\mf{sl}(n,\C)$. Then by \cite[Lemma 5.2]{rubenthaler}, we have that $\mf{a} \cap \mf{b} = Z_{\mf{a}} = Z_{\mf{b}} =: Z$, that $\mf{a} = Z \oplus [\mf{a},\mf{a}]$ and $\mf{b} = Z \oplus [\mf{b}, \mf{b}]$, and that $\ell_Z = \mf{z}_{\mf{sl}(n,\C)}(Z)$ is a Levi subalgebra of $\mf{sl}(n,C)$. By Lemma \ref{lem:sl-levis}, we can write 
$$\ell_Z = \mf{sl}(n_1,\C) \oplus \cdots \oplus \mf{sl}(n_r,\C) \oplus \C^{r-1}$$
for some $n_1, \ldots , n_r$ so that $\sum_i n_i = n$. It follows that $Z = \C^{r-1}$. Using \cite[Lemma 5.2]{rubenthaler} yet again, we get that $([\mf{a},\mf{a}], [\mf{b},\mf{b}])$ is a semisimple dual pair in $[\ell_Z, \ell_Z ] = \mf{sl}(n_1,\C) \oplus \cdots \oplus \mf{sl}(n_r, \C)$. By \cite[Lemma 5.3]{rubenthaler} and Rubenthaler's classification of the semisimple dual pairs in $\mf{sl}(n,\C)$, the semisimple dual pairs in $[\ell_Z, \ell_Z]$ are exactly those of the form
$$(\mf{sl}(k_1,\C)\oplus \cdots \oplus \mf{sl}(k_r,\C), \; \mf{sl}(\ell_1,\C) \oplus \cdots \oplus \mf{sl}(\ell_r,\C) ),$$
where $\sum_i k_i \ell_i = \sum_i n_i = n$. It follows that $[\mf{a}, \mf{a}]$ and $[\mf{b},\mf{b}]$ are of the form
$$[\mf{a}, \mf{a}] = \mf{sl}(k_1,\C)\oplus \cdots \oplus \mf{sl}(k_r,\C) \hspace{.5cm} \text{ and } \hspace{.5cm} [\mf{b}, \mf{b}] = \mf{sl}(\ell_1,\C) \oplus \cdots \oplus \mf{sl}(\ell_r,\C) ),$$
and hence that $\mf{a}$ and $\mf{b}$ are of the form 
$$\mf{a} = \mf{sl}(k_1,\C)\oplus \cdots \oplus \mf{sl}(k_r,\C) \oplus \C^{r-1} \hspace{.25cm} \text{ and } \hspace{.25cm} \mf{b} = \mf{sl}(\ell_1,\C) \oplus \cdots \oplus \mf{sl}(\ell_r,\C) \oplus \C^{r-1} ,$$
completing the proof.
\end{proof}

\subsection{Dual pairs in \texorpdfstring{$\mf{gl}(n,\C)$}{gl(n,C)}} \label{subsec:gl-LA}

\begin{thm} \label{thm:dual-pairs-in-gl}
The dual pairs in $\mf{gl}(n,\mathbb{C})$ are exactly the pairs of subalgebras of the form
$$\left ( \bigoplus_{i=1}^r \mf{gl}(k_i,\C) , \; \bigoplus_{i=1}^r \mf{gl}(\ell_i,\C) \right ),$$
where $k_i, \ell_i >0$, and where $\sum_i k_i \ell_i = n$.
\end{thm}

\begin{proof}
By Proposition \ref{prop:1}, Remark \ref{rmk:connected}, and Theorem \ref{thm:dual-pairs-in-GL}, all of the pairs of this form are indeed dual pairs in $\mf{gl}(n,\C)$. On the other hand, suppose $(\mf{g}_1, \mf{g}_2)$ is a dual pair in $\mf{gl}(n,\C)$. Set $G_2 = Z_G( \langle \exp \mf{g}_1 \rangle)$, and $G_1 = Z_G (G_2)$. Then by Proposition \ref{prop:2} and Theorem \ref{thm:dual-pairs-in-GL}, $(G_1, G_2)$ is of the form
$$(G_1, G_2) = \left ( \prod_{i=1}^{r} GL(k_i), \; \prod_{i=1}^{r} GL(\ell_i) \right ),$$
where $\sum_i k_i \ell_i = n$, and where $\Lie (G_i) = \mf{g}_i$. It follows that $(\mf{g}_1, \mf{g}_2)$ is of the form described in the proposition statement, completing the proof. 
\end{proof}

This result, obtained here from our analysis of Lie group dual pairs and the correspondence between Lie group and Lie algebra dual pairs, can also be obtained from a Lie algebras context alone (using Rubenthaler's reduction and our analysis of dual pairs in $\mf{sl}(n,\C)$).

\begin{proof}[Alternative proof of Theorem \ref{thm:dual-pairs-in-gl}]
To start, note that we can write 
$$\mf{gl}(n,\C) = [\mf{gl}(n,\C), \mf{gl}(n,\C) ] \oplus Z_{\mf{gl}(n,\C)} = \mf{sl}(n,\C) \oplus \C .$$
Therefore, by \cite[Lemma 5.1]{rubenthaler} and Theorem \ref{thm:sln-dps}, the dual pairs in $\mf{gl}(n,\C)$ are exactly the pairs of the form 
\begin{align*}
(\mf{sl}(k_1,\C)\oplus \cdots \oplus \mf{sl}(k_r,\C) \oplus \C^{r-1} \oplus \C, & \; \mf{sl}(\ell_1,\C) \oplus \cdots \oplus \mf{sl}(\ell_r,\C) \oplus \C^{r-1} \oplus \C ) \\
= ( \mf{gl}(k_1,\C) \oplus \cdots \oplus \mf{gl}(k_r,\C),  & \; \mf{gl}(\ell_1,\C) \oplus \cdots \oplus \mf{gl}(\ell_r,\C)),
\end{align*}
where $\sum_i k_i \ell_i = n$.
\end{proof}

In Subsections \ref{subsec:sl-LA} and \ref{subsec:gl-LA}, we demonstrated how to verify our classifications of $\mf{sl}(n,\C)$ and $\mf{gl}(n,\C)$ dual pairs using Rubenthaler's work. While the Lie algebra dual pair results in the remainder of the paper can also be checked against Rubenthaler's work, such verifications are much more difficult since not all dual pairs in $\mf{sp}(2n,\C)$ and $\mf{so}(n,\C)$ are ``$S$-irreducible" \cite[Subsections 5.7, 6.4, 6.6, 6.7]{rubenthaler}.

\subsection{Bijections of dual pairs}

The results in this section can be nicely related as follows:

\begin{thm}
The following sets are in bijection:
\begin{enumerate}[label = (\roman*)]
\item dual pairs in $GL(n,\mathbb{C})$;
\item dual pairs in $SL(n,\mathbb{C})$;
\item connected dual pairs in $PGL(n,\mathbb{C})$;
\item dual pairs in $\mf{gl}(n,\mathbb{C})$;
\item dual pairs in $\mf{sl}(n,\mathbb{C})$. 
\end{enumerate}
\end{thm}

\begin{proof}
\noindent \underline{$(i) \leftrightarrow (ii)$:} The bijection $(i) \leftrightarrow (ii)$ follows directly from Theorem \ref{thm:subgroup-bijection}, where $\rightarrow$ is given by intersection with $SL(n,\mathbb{C})$ and $\leftarrow$ is given by multiplication by $\mathbb{C}^{\times}$.\\  

\noindent \underline{$(i) \leftrightarrow (iii)$:} As explained in the proof of Proposition \ref{prop:connected-connected}, we have that a subgroup $G$ in $PGL(n,\mathbb{C})$ is connected if and only if its preimage in $GL(n,\mathbb{C})$ is connected. Since all dual pairs in $GL(n,\mathbb{C})$ are connected (by Corollary \ref{cor:connected}), Proposition \ref{prop:GL-PGL-descent} and Proposition \ref{prop:connected-connected} therefore give that the dual pairs in $GL(n,\mathbb{C})$ are in bijection with the connected dual pairs in $PGL(n,\mathbb{C})$. In this bijection, $(i) \rightarrow (iii)$ is given by $(G_1, G_2) \mapsto (p(G_1), p(G_2))$, and $(iii) \leftarrow (i)$ is given by $(\overline{G}_1, \overline{G}_2) \mapsto ( p^{-1}(\overline{G}_1), p^{-1}(\overline{G}_2) )$. \\

\noindent \underline{$(i) \leftrightarrow (iv)$:} Recall that the dual pairs in $GL(n,\C)$ are exactly the pairs of groups of the form
$$(GL(k_1,\C) \times \cdots \times GL(k_r,\C), \; GL(\ell_1,\C) \times \cdots \times GL(\ell_r,\C) ),$$
where $\sum_i k_i \ell_i = n$ (by Theorem \ref{thm:dual-pairs-in-GL}). Similarly, by Theorem \ref{thm:dual-pairs-in-gl}, the dual pairs in $\mf{gl}(n,\C)$ are exactly the pairs of subalgebras of the form 
$$(\mf{gl}(k_1,\C) \oplus \cdots \oplus \mf{gl}(k_r,\C) ,\; \mf{gl}(\ell_1,\C) \oplus \cdots \oplus \mf{gl}(\ell_r,\C) ),$$
where $k_i,\ell_i >0$, and where $\sum_i k_i \ell_i = n$. Therefore, there is clearly a bijection $(i) \leftrightarrow (iv)$, where $\rightarrow$ is given by $\Lie (\cdot )$ and $\leftarrow$ is given by $\exp (\cdot )$. Also, once can check that these align with the maps given by Propositions \ref{prop:1} and \ref{prop:2}.\\

\noindent \underline{$(iv) \leftrightarrow (v)$:} This follows from Theorems \ref{thm:sln-dps} and \ref{thm:dual-pairs-in-gl}, where $\rightarrow$ is given by intersection with $\mf{sl}(n,\C)$ and $\leftarrow$ is given by $\oplus \C$.
\end{proof}

\begin{remark}
Recall from the proof of Theorem \ref{thm:dual-pairs-in-SL} that $GL(n,\C) = \C^{\times} \cdot SL(n,\C)$. Therefore, for any $g = zs \in GL(n,\C)$ (with $z \in \C^{\times}$ and $s \in SL(n,\C)$), 
$$g \cdot \C^{\times} = (sz) \cdot \C^{\times} = s \cdot \C^{\times}.$$
In this way, we see that $PGL(n,\C) = PSL(n,\C)$. This justifies our exclusion of $PSL(n,\C)$ from the list of complex projective classical groups that are under consideration.
\end{remark}

\section{Dual pairs in \texorpdfstring{$Sp(2n,\C)$}{Sp(2n,C)}, \texorpdfstring{$PSp(2n,\C)$}{PSp(2n,C)}, and \texorpdfstring{$\mf{sp}(2n,\C)$}{sp(2n,C)}} \label{sec:sp}

\subsection{Dual pairs in \texorpdfstring{$Sp(2n,\C)$}{Sp(2n,C)}} \label{subsec:Sp}

The following lemma sets us up to apply an argument analogous to the proof of Proposition \ref{prop:dual-pairs-GL} in the context of $Sp(2n,\C)$, which will allow us classify the dual pairs in $Sp(2n,\C)$.

\begin{lemma} \label{lem:cap-Sp(U)}
Let $U$ be a finite-dimensional complex symplectic vector space of even dimension. Let $H$ be a complex reductive algebraic group, and let $\varphi \colon H \rightarrow Sp(U)$ be an algebraic symplectic representation of $H$ with nonisomorphic irreducible subrepresentations 
$$\{ V_{\gamma} \}_{\gamma} = \{ V_{\mu}, V_{\nu}, V_{\lambda} \}_{\mu \not \simeq \mu^*, \; \substack{ \nu \simeq \nu^* \\ \text{orthog.}}, \; \substack{\lambda \simeq \lambda^* \\ \text{sympl.}} }.$$ 
Set $W_{\gamma}:= \Hom_H(V_{\gamma},U)$. Then
\begin{align}
    \left ( \prod_{\gamma} GL(V_{\gamma}) \right ) \cap Sp(U) &= \prod_{\mu \not \simeq \mu^*} GL(V_{\mu}) \prod_{\substack{\nu \simeq \nu^* \\ \text{orthog.}}} O(V_{\nu}) \prod_{\substack{\lambda \simeq \lambda^* \\ \text{sympl.}}} Sp(V_{\lambda}), \text{ and } \label{eq:align-1} \\
\left ( \prod_{\gamma} GL(W_{\gamma}) \right ) \cap Sp(U) &= \prod_{\mu \not \simeq \mu^*} GL(W_{\mu}) \prod_{\substack{\nu \simeq \nu^* \\ \text{orthog.}}} Sp(W_{\nu}) \prod_{\substack{\lambda \simeq \lambda^* \\ \text{sympl.}}} O(W_{\lambda}). \label{eq:align-2}
\end{align}
\end{lemma}

\begin{proof}
To start, note that Schur's lemma gives the decomposition
$$U \simeq \Bigg ( \bigoplus_{\mu \not \simeq \mu^*} V_{\mu} \otimes W_{\mu} \Bigg ) \oplus \Bigg ( \bigoplus_{\substack{\nu \simeq \nu^* \\ \text{orthog.} }} V_{\nu} \otimes W_{\nu} \Bigg ) \oplus \Bigg ( \bigoplus_{\substack{\lambda \simeq \lambda^* \\ \text{sympl.} }} V_{\lambda} \otimes W_{\lambda} \Bigg ) .$$
Our strategy in this proof is to understand the structure on each of these summands and tensor factors that is induced by the symplectic structure of $U$. Then (\ref{eq:align-1}) and (\ref{eq:align-2}) will follow from a consideration of which elements of $\prod_{\gamma} GL(V_{\gamma})$ and $\prod_{\gamma} GL(W_{\gamma})$ preserve this substructure. 

Since $U$ is symplectic and finite-dimensional, we have that $U \simeq U^*$, which allows us to write
$$U \simeq \Bigg ( \bigoplus_{(\mu, \mu^*)} (V_{\mu} \otimes W_{\mu}) \oplus (V_{\mu^*} \otimes W_{\mu^*}) \Bigg ) \oplus \Bigg ( \bigoplus_{\substack{\nu \simeq \nu^* \\ \text{orthog.} }} V_{\nu} \otimes W_{\nu} \Bigg )\oplus \Bigg ( \bigoplus_{\substack{\lambda \simeq \lambda^* \\ \text{sympl.} }} V_{\lambda} \otimes W_{\lambda} \Bigg ).$$
Now, note that for each irreducible representation $V_{\gamma}$ of $H$, $V_{\gamma^*}$ is also irreducible, so
\[
    \dim (\Hom_H(V_{\gamma},V_{\gamma^*})) = \Bigg \{\begin{array}{lr}
        1 & \text{if } V_{\gamma} \cong V_{\gamma^*}\\
        0 & \text{if } V_{\gamma} \not \cong V_{\gamma^*}
        \end{array}.
  \]
Note also that $H$-invariant bilinear forms are in one-to-one correspondence with the elements of $\Hom_H(V_{\gamma},V_{\gamma^*})$. Therefore, the $V_{\mu}$'s do not admit $H$-invariant bilinear forms, whereas the remaining irreducibles each inherit an $H$-invariant bilinear form from $U$ (which must be either symplectic or orthogonal). It is not hard to see that, by extension, the $V_{\mu} \otimes W_{\mu}$'s do not admit $H$-invariant bilinear forms, whereas the $V_{\nu} \otimes W_{\nu}$'s and $V_{\lambda} \otimes W_{\lambda}$'s inherit $H$-invariant symplectic bilinear forms from $U$. Moreover, since $W_{\gamma} \simeq W_{\gamma^*}$ for $\gamma \simeq \gamma^*$, each $W_{\gamma}$ admits an $H$-invariant bilinear form, which must be symplectic if the form on $V_{\gamma}$ is orthogonal and must be orthogonal if the form on $V_{\gamma}$ is symplectic (since the symplectic form on $V_{\gamma} \otimes W_{\gamma}$ can be obtained as the product of the forms on $V_{\gamma}$ and $W_{\gamma}$).

Although each $V_{\mu} \otimes W_{\mu}$ for $\mu \not \simeq \mu^*$ does not admit an $H$-invariant bilinear form, we claim that each $(V_{\mu} \otimes W_{\mu}) \oplus (V_{\mu^*} \otimes W_{\mu^*})$ admits an $H$-invariant symplectic form. Indeed, write $E:= V_{\mu} \otimes W_{\mu}$ and $E^* := V_{\mu^*} \otimes W_{\mu^*}$. Then for $e,e' \in E$ and $\mathcal{E}, \mathcal{E}' \in E^*$, we see that
$$\langle (e,\mathcal{E}), \; (e',\mathcal{E}') \rangle := \mathcal{E}'(e)-\mathcal{E}(e')$$
defines an $H$-invariant symplectic bilinear form on $E \oplus E^*$.

Finally, it is clear that an element of $\prod_{\gamma} GL(V_{\gamma})$ preserves the symplectic bilinear form on $U$ if and only if it preserves the induced bilinear form on each $V_{\gamma}$. Similarly, an element of $\prod_{\gamma} GL(W_{\gamma})$ preserves the symplectic bilinear form on $U$ if and only if it preserves the induced bilinear form on each $W_{\gamma}$. This established \eqref{eq:align-1} and \eqref{eq:align-2}, completing the proof. 
\end{proof}

\begin{example} \label{ex:symplectic-form}
Let $H$ be a complex reductive algebraic group, and let $\varphi \colon H \rightarrow Sp(U)$ be an algebraic symplectic representation of $H$ with nonisomorphic irreducible subrepresentations $V_1$, $V_1^*$, and $V_2$. Write $U \simeq V_1 \otimes W_1 \oplus V_1^* \otimes W_1^* \oplus V_2 \otimes W_2$, suppose that $V_1 \not \simeq V_1^*$ as representations, that $V_2$ is orthogonal, and that $\dim V_i = \dim W_i = 2$. The matrix
\[
\Omega = \begin{pmatrix} & I_4 & & & & \\-I_4 & & & & & \\ & & & 1 & & \\ & & -1 & & & \\ & & & & & 1 \\ & & & & -1 & \end{pmatrix}
\]
defines a symplectic form on $U$ (so that any matrix $g \in Sp(U)$ satisfies $g^t \Omega g = \Omega$). The symplectic form $\Omega$ induces the symplectic form 
\[
\Omega_1 := \begin{pmatrix} & I_4 \\ -I_4 & \end{pmatrix}
\]
on $V_1 \otimes W_1 \oplus V_1^* \otimes W_1^*$. To see that $A = \left ( \begin{smallmatrix} a & b \\ c & d \end{smallmatrix} \right ) \in GL(V_1)$ and $B = \left (  \begin{smallmatrix} e & f \\ g & h \end{smallmatrix} \right ) \in GL(W_1)$ preserve $\Omega_1$, we consider the images of $A$ and $B$ under the embeddings of $GL(V_1)$ and $GL(W_1)$ into $GL(V_1 \otimes W_1 \oplus V_1^* \otimes W_1^*)$:
\begin{center}
\begin{tabular}{r l l}
    $A$ &$\xmapsto{\iota} \begin{pmatrix} aI_2 & bI_2 \\ cI_2 & dI_2 \end{pmatrix}$ & $\xmapsto{\tau} \begin{pmatrix} \iota(A) & \\ & (\iota(A)^t)^{-1} \end{pmatrix}$ \\
$B$ & $\xmapsto{\kappa} \begin{pmatrix} B & \\ & B \end{pmatrix}$ & $\xmapsto{\tau} \begin{pmatrix} \kappa(B) & \\ & (\kappa(B)^t)^{-1} \end{pmatrix}$
\end{tabular}
\end{center}
It is straightforward to check that $\tau(\iota(A))$ and $\tau(\kappa(B))$ preserve $\Omega_1$, and hence lie in $Sp(V_1 \otimes W_1 \oplus V_1^* \otimes W_1^*)$. Additionally, the symplectic form $\Omega$ induces the symplectic form
\[
\Omega_2 := \begin{pmatrix} & 1 & & \\ -1 & & & \\ & & & 1 \\ & & -1 & \end{pmatrix} 
\]
on $Sp(V_2 \otimes W_2)$. Consider the images of $A' = \left ( \begin{smallmatrix} a' & b' \\ c' & d' \end{smallmatrix} \right ) \in GL(V_2)$ and $B' = \left ( \begin{smallmatrix} e' & f' \\ g' & h' \end{smallmatrix} \right ) \in GL(W_2)$ under the embeddings of $GL(V_2)$ and $GL(W_2)$ into $GL(V_2 \otimes W_2)$:
\begin{center}
\begin{tabular}{r l l}
    $A'$ & $\xmapsto{\iota} \begin{pmatrix} a'I_2 & b'I_2 \\ c'I_2 & d'I_2 \end{pmatrix}$ & $\in GL(V_2 \otimes W_2)$ \\
    $B'$ & $\xmapsto{\kappa} \begin{pmatrix} B' & \\ & B' \end{pmatrix}$ & $\in GL(V_2 \otimes W_2)$
\end{tabular}
\end{center}
It is straightforward to check that $\iota(A')$ preserves $\Omega_2$ if and only if $A' \in O(V_2)$ and that $\kappa(B')$ preserves $\Omega_2$ if and only if $B' \in Sp(W_2)$. It follows that in this case, 
\begin{align*}
    \left ( GL(V_1) \times GL(V_1^*) \times GL(V_2) \right ) \cap Sp(U) &= GL(V_1) \times GL(V_1^*) \times O(V_2), \; \text{ and} \\
    \left ( GL(W_1) \times GL(W_1^*) \times GL(W_2) \right ) \cap Sp(U) &= GL(W_1) \times GL(W_1^*) \times Sp(W_2),
\end{align*}
as Lemma \ref{lem:cap-Sp(U)} would suggest.
\end{example}

\begin{prop} \label{prop:dual-pairs-in-Sp}
Let $U$ be a finite-dimensional complex symplectic vector space of even dimension, and let $\mu$, $\nu$, and $\lambda$ be as in Lemma \ref{lem:cap-Sp(U)}. Then the dual pairs of $Sp(U)$ are exactly the pairs of groups of the form
$$\left ( \prod_{\mu \not \simeq \mu^*} GL(V_{\mu}) \prod_{\substack{\nu \simeq \nu^* \\ \text{orthog.}}} O(V_{\nu}) \prod_{\substack{\lambda \simeq \lambda^* \\ \text{sympl.}}} Sp(V_{\lambda}), \; \prod_{\mu \not \simeq \mu^*} GL(W_{\mu}) \prod_{\substack{\nu \simeq \nu^* \\ \text{orthog.}}} Sp(W_{\nu}) \prod_{\substack{\lambda \simeq \lambda^* \\ \text{sympl.}}} O(W_{\lambda})  \right ),$$
where 
$$U = \Bigg ( \bigoplus_{\mu \not \simeq \mu^*} ((V_{\mu} \otimes W_{\mu}) \oplus (V_{\mu}^* \otimes W_{\mu}^*)) \Bigg ) \oplus \Bigg ( \bigoplus_{ \substack{ \nu \simeq \nu^* \\ \text{orthog.} }} V_{\nu} \otimes W_{\nu} \Bigg ) \oplus \Bigg ( \bigoplus_{\substack{\lambda \simeq \lambda^* \\ \text{sympl.} }} V_{\lambda} \otimes W_{\lambda} \Bigg ) $$
is a vector space decomposition of $U$ with $\dim V_{\lambda}$ even and $\dim W_{\nu}$ even. 
\end{prop}

\begin{proof}
Let $\varphi \colon H \rightarrow Sp(U)$ be an algebraic symplectic representation with nonisomorphic irreducible subrepresentations $\{ (\gamma, V_\gamma) \}$. Set $W_{\gamma}:= \Hom_H(V_{\gamma},U)$. Write 
\begin{align*}
 G_1 &:= \prod_{\mu \not \simeq \mu^*} GL(V_{\mu}) \prod_{\substack{\nu \simeq \nu^* \\ \text{orthog.}}} O(V_{\nu}) \prod_{\substack{\lambda \simeq \lambda^* \\ \text{sympl.}}} Sp(V_{\lambda}), \text{ and } \\
 G_2 &:= \prod_{\mu \not \simeq \mu^*} GL(W_{\mu}) \prod_{\substack{\nu \simeq \nu^* \\ \text{orthog.}}} Sp(W_{\nu}) \prod_{\substack{\lambda \simeq \lambda^* \\ \text{sympl.}}} O(W_{\lambda}).   
\end{align*}

\noindent \textit{Step 1: $Z_{Sp(U)} (G_1) = G_2 = Sp(U)^{\varphi}$.} \\

Let $t \in Z_{Sp(U)}(G_1)$. By Lemma \ref{lem:cap-Sp(U)}, $\varphi (H) \subseteq G_1$, so we have that $t$ commutes with any element of $\varphi (H)$. In other words, $t$ is $H$-linear. Applying Schur's lemma in the same way as in Step 1 of the proof of Proposition \ref{prop:dual-pairs-GL} gives that $t \in \prod_{\gamma} GL(W_{\gamma}) \cap Sp(U) = G_2$, where $\gamma$ ranges over the nonisomorphic irreducible subrepresentations of $\varphi$, and where we have used Lemma \ref{lem:cap-Sp(U)}. It follows that $Z_{Sp(U)}(G_1) \subseteq G_2$. On the other hand, it is clear that $G_2 \subseteq Z_{Sp(U)}(G_1)$. It also follows from this argument that $G_2$ consists exactly of the $H$-linear elements of $Sp(U)$, completing Step 1.\\

\noindent \textit{Step 2: Each $W_{\gamma}$ is an irreducible representation of $G_2$ with multiplicity space $V_{\gamma}$.} \\

Consider the representation $\rho \colon G_2 \hookrightarrow GL(U)$. As in the proof of Proposition \ref{prop:dual-pairs-GL}, we will show that the $W_{\gamma}$'s are precisely the nonisomorphic irreducible subrepresentations of $\rho$, and that $W_{\gamma}$ has multiplicity space $V_{\gamma}$. To see this, we start by noting that each $W_{\gamma}$ is a subrepresentation of $U$, since $\rho(g)w \in W_{\gamma}$ for any $g \in G_2$ and $w \in W_{\gamma}$. Moreover, each $W_{\gamma}$ is irreducible by Lemma \ref{lem:standard-irreps}. We therefore obtain two decompositions of $U$, giving
$$U \simeq \bigoplus V_{\gamma} \otimes \Hom_H(V_{\gamma}, U) \simeq \bigoplus W_{\gamma} \otimes \Hom_{G_2} (W_{\gamma}, U).$$
Since $W_{\gamma} = \Hom_H(V_{\gamma}, U)$, this shows that $V_{\gamma} = \Hom_{G_2} (W_{\gamma}, U)$, completing Step 2. \\

\noindent \textit{Step 3: $Z_{Sp(U)} (G_2) = G_1$.} \\

Step 2 shows that we can repeat Step 1 with $G_2$ in place of $H$ and with the roles of $W_{\gamma}$ and $V_{\gamma}$ reversed. Doing so gives that $Z_{Sp(U)} (G_2) = G_1$, as desired. \\

\noindent \textit{Step 4: All dual pairs of $Sp(U)$ are of this form.} \\

It was shown in Step 1 that $G_2 = Sp(U)^{\varphi}$. Step 4 therefore follows from Remark \ref{rmk:all pairs}, completing the proof.
\end{proof}

\begin{thm} \label{thm:dual-pairs-in-Sp}
The dual pairs of $Sp(2n,\C)$ are exactly the pairs of groups of the form
\begin{equation*}
    \left ( \prod_{\mu} GL(k_{\mu},\C) \prod_{ \nu } O(k_{\nu},\C) \prod_{\lambda} Sp(2k_{\lambda},\C), \; \prod_{\mu} GL(\ell_{\mu},\C) \prod_{ \nu } Sp(2\ell_{\nu},\C) \prod_{\lambda} O(\ell_{\lambda},\C) \right ),
\end{equation*}
where $k_{\gamma} = 0$ if and only if $\ell_{\gamma} = 0$, and where $\sum_{\mu} k_{\mu} \ell_{\mu} + \sum_{\nu} k_{\nu}\ell_{\nu} + \sum_{\lambda} k_{\lambda} \ell_{\lambda} = n$.
\end{thm}

\begin{proof}
This follows from the proof of Proposition \ref{prop:dual-pairs-in-Sp} by taking 
$$H = \prod_{\mu} GL(V_{\mu}) \prod_{ \nu } O(V_{\nu}) \prod_{\lambda} Sp(V_{\lambda})$$ 
(with $\dim V_{\gamma} = k_{\gamma}$) and applying Lemma \ref{lem:standard-irreps}.
\end{proof}

\subsection{Dual pairs in \texorpdfstring{$PSp(2n,\C)$}{PSp(2n,C)}} \label{subsec:PSp}

Let $p\colon Sp(n,\C) \rightarrow PSp(n,\C)$ be the canonical projection.

\begin{prop} \label{prop:dual-pairs-in-PSp}
For $n$ even, define
$$\displaystyle \mathcal{S}(n,\C) := \left \langle Sp(n,\C), \; \begin{pmatrix} & I_{n} \\ -I_n & \end{pmatrix} \right \rangle \subseteq Sp(2n,\C).$$
Let $(G_1, G_2)$ be the $Sp(2n,\C)$ dual pair
\begin{equation*}
  \left ( \prod_{\mu} GL(k_{\mu},\C) \prod_{ \nu } O(k_{\nu},\C) \prod_{\lambda} Sp(2k_{\lambda},\C), \; \prod_{\mu} GL(\ell_{\mu},\C) \prod_{ \nu } Sp(2\ell_{\nu},\C) \prod_{\lambda} O(\ell_{\lambda},\C) \right ),
\end{equation*}
where $k_{\gamma} = 0$ if and only if $\ell_{\gamma} = 0$, and where $\sum_{\mu} k_{\mu} \ell_{\mu} + \sum_{\nu} k_{\nu}\ell_{\nu} + \sum_{\lambda} k_{\lambda} \ell_{\lambda} = n$. 
\begin{enumerate}[label = (\roman*)]
\item If $n$ is even, then $(p (O(2,\C)), p (\mathcal{S}(n,\C)))$ is a dual pair in $PSp( 2n,\C )$;
\item If $2 \not \in \{ k_{\nu} \}_{\nu} \cup \{ \ell_{\lambda} \}_{\lambda}$, then $(p(G_1), p(G_2))$ is a dual pair in $PSp(2n,\C)$;
\item If $\vert \{ \mu , \nu , \lambda \} \vert \geq 2$, then $(p(G_1), p(G_2))$ is a dual pair in $PSp(2n,\C)$.
\end{enumerate}
\end{prop}

\begin{proof}[Proof of (i)]
We first show $Z_{PSp(2n,\C)} (p(O(2,\C))) = p (\mathcal{S}(n,\C))$. To this end, recall that
$$O(2,\C) = \left \langle \begin{pmatrix} 1 & \\ & -1 \end{pmatrix} , \; \left \{ \begin{pmatrix} z_1 & z_2 \\ -z_2 & z_1 \end{pmatrix} \; : \; z_1^2 + z_2^2 =1 \right \} \right \rangle,$$
and let 
$$M \in p^{-1}(Z_{PSp(2n,\C)} (p (O(2,\C)))).$$ 
Then writing out the entry-wise implications of the relation
\begin{equation} \label{eq:plus-or-minus}
    M \begin{pmatrix} I_n & \\ & -I_n \end{pmatrix} = \pm \begin{pmatrix} I_n & \\ & -I_n \end{pmatrix} M
\end{equation}
gives that $M$ has the block form $\left ( \begin{smallmatrix} M_{11} & \\ & M_{22} \end{smallmatrix} \right )$ if the sign in (\ref{eq:plus-or-minus}) is positive, and has the $2 \times 2$ block form $\left ( \begin{smallmatrix} & M_{12} \\ M_{21} & \end{smallmatrix} \right )$ if the sign in (\ref{eq:plus-or-minus}) is negative (where $M_{11}, M_{12}, M_{21}, M_{22}$ are $n \times n$ matrices). Subsequently writing out the entry-wise implications of the relation
\begin{equation} \label{eq:plus-or-minus2}
    M \begin{pmatrix} z_1 I_n & z_2 I_n \\ - z_2 I_n & z_1 I_n \end{pmatrix} = \pm \begin{pmatrix} z_1 I_n & z_2 I_n \\ - z_2 I_n & z_1 I_n \end{pmatrix} M
\end{equation}
further gives that $M$ has the block form $\left ( \begin{smallmatrix} M_{11} & \\ & M_{11} \end{smallmatrix} \right )$ if the signs in (\ref{eq:plus-or-minus}) and (\ref{eq:plus-or-minus2}) are both positive, and that $M$ has the block form $\left ( \begin{smallmatrix} & M_{12} \\ -M_{12} & \end{smallmatrix} \right )$ if the signs in (\ref{eq:plus-or-minus}) and (\ref{eq:plus-or-minus2}) are negative and positive, respectively. Moreover, we see that the remaining sign combinations are impossible. Finally, requiring that $M^t \Omega M = \Omega$ shows that $M_{11}$ and $M_{12}$ are in $Sp(n,\C)$. Consequently, 
$$Z_{PSp(2n,\C)}(p (O(2,\C))) \subseteq p (\mathcal{S}(n,\C)).$$ 
On the other hand, it is straightforward to check containment in the other direction. It follows that $Z_{PSp(2n,\C)}(p (O(2,\C))) = p (\mathcal{S}(n,\C))$, as desired.

We have left to show that $Z_{PSp(2n,\C)} (p (\mathcal{S}(n,\C))) = p (O(2,\C))$. To this end, let $N \in p^{-1}(Z_{PSp(2n,\C)}(p(\mathcal{S}(n,\C))))$. Then, in particular,
\begin{equation} \label{eq:plus-or-minus3}
    N \begin{pmatrix} A & \\ & A \end{pmatrix} = \pm \begin{pmatrix} A & \\ & A \end{pmatrix} N \hspace{.25cm} \text{ for all } A \in Sp(n,\C).
\end{equation}
Writing out the entry-wise implications of (\ref{eq:plus-or-minus3}) gives that $N$ is of the form $N = \left ( \begin{smallmatrix} aI_n & bI_n \\ cI_n & dI_n \end{smallmatrix} \right )$ for some $a,b,c,d \in \mathbb{C}$. The relation $N^t \Omega N = \Omega$ further gives that $N\in O(2,\C)$, and hence that
$$Z_{PSp(2n,\C)}(p(\mathcal{S}(n,\C))) \subseteq p (O(2,\C)).$$ 
On the other hand, it is straightforward to check that both $\left ( \begin{smallmatrix} I_n & \\ & -I_n \end{smallmatrix} \right )$ and $\left ( \begin{smallmatrix} z_1I_n & z_2 I_n \\ -z_2I_n & z_1I_n \end{smallmatrix} \right )$ commute with $\left ( \begin{smallmatrix} & I_n \\ -I_n &  \end{smallmatrix} \right )$ (up to $\pm 1$). Therefore, $Z_{PSp(2n,\C)}(p (\mathcal{S}(n,\C))) = p (O(2,\C))$. 
\end{proof}

\begin{proof}[Proof of (ii)]
Suppose that $2 \not \in \{ k_{\nu} \}_{\nu} \cup \{ \ell_{\lambda} \}_{\lambda}$. Then Proposition \ref{prop:connected-dp-descend} and Lemma \ref{lem:standard-irreps} give that $(p(G_1), p(G_2))$ is a dual pair.
\end{proof}

\begin{proof}[Proof of (iii)]
Suppose that $\vert \{ \mu , \nu , \lambda \} \vert \geq 2$. It is clear that 
$$p(G_2) \subseteq Z_{PSp(2n,\C)}(p(G_1)) \hspace{.25cm} \text{ and } \hspace{.25cm} Z_{PSp(2n,\C)}( p(G_1)) \subseteq Z_{PSp(2n,\C)}( p(G_1^{\circ})).$$ 
Write 
$$G_1 = \prod_{\substack{ \nu \text{ with}\\ k_{\nu}=2}} O(k_{\nu},\C) \prod_{\substack{ \nu \text{ with} \\ k_{\nu} \neq 2}} O(k_{\nu},\C) \prod_{\mu} GL(k_{\mu},\C) \prod_{\lambda} Sp(2k_{\lambda},\C).$$
Applying part $(i)$, we see that
\begin{align*}
Z_{PSp(2n,\C)}( p(G_1^{\circ})) & \subseteq p \left (\prod_{\substack{ \nu \text{ with}\\ k_{\nu}=2}} \mathcal{S}(2\ell_{\nu},\C) \prod_{\substack{\nu \text{ with}\\ k_{\nu} \neq 2}} Sp(2\ell_{\nu},\C) \prod_{\mu} GL(\ell_{\mu},\C) \prod_{\lambda} O(\ell_{\lambda},\C) \right ).
\end{align*}
However, note that the commutator of $\left ( \begin{smallmatrix} & I_{2\ell_{\nu}} \\ -I_{2\ell_{\nu}} \end{smallmatrix} \right ) \in \mathcal{S}(2\ell_{\nu},\C)$ with certain elements of $O(2,\C)$ equals $-1$. Therefore, if $\vert \{ \mu, \nu , \lambda \} \vert \geq 2$, then for any element 
$$\diag \left ( \left ( \begin{smallmatrix} & I_{2\ell_{\nu}} \\ -I_{2\ell_{\nu}} \end{smallmatrix} \right ), X \right ) \in \prod_{\substack{ \nu \text{ with}\\ k_{\nu}=2}} \mathcal{S}(2\ell_{\nu},\C) \prod_{\substack{\nu \text{ with}\\ k_{\nu} \neq 2}} Sp(2\ell_{\nu},\C) \prod_{\mu} GL(\ell_{\mu},\C)  \prod_{\lambda} O(\ell_{\lambda},\C),$$
there exists $\diag (g,X') \in G_1$ with $g \in O(2,\C)$ such that 
$$ \begin{pmatrix} \begin{smallmatrix} & I_{2\ell_{\nu}} \\ -I_{2\ell_{\nu}} \end{smallmatrix} & \\ & X \end{pmatrix} \begin{pmatrix} g & \\ & X' \end{pmatrix} \begin{pmatrix} \begin{smallmatrix} & I_{2\ell_{\nu}} \\ -I_{2\ell_{\nu}} \end{smallmatrix} & \\ & X \end{pmatrix}^{-1} \begin{pmatrix} g & \\ & X' \end{pmatrix}^{-1} = \begin{pmatrix} -1 & \\ & 1 \end{pmatrix} .$$
It follows that 
$$Z_{PSp(2n,\C)} ( p(G_1) ) \subseteq  p \left ( \prod_{\nu} Sp(2\ell_{\nu},\C) \prod_{\mu} GL(\ell_{\mu},\C) \prod_{\lambda} O(\ell_{\lambda},\C) \right ) = p(G_2) ,$$
and hence that $Z_{PSp(2n,\C)}(p(G_1)) = p(G_2)$. The equality $Z_{PSp(2n,\C)} (p(G_2)) = p(G_1)$ follows in the same way, completing the proof.
\end{proof}

Classifying the remaining dual pairs in $PSp(2n,\C)$ appears to be a difficult question (related to the problem of classifying the disconnected dual pairs in $PGL(2n,\C)$) and is beyond the scope of this paper.

\subsection{Dual pairs in \texorpdfstring{$\mf{sp}(2n,\C)$}{sp(2n,C)}} \label{subsec:sp-LA}

\begin{thm} \label{thm:dual-pairs-in-sp-LA}
The dual pairs of $\mf{sp}(2n,\C)$ are exactly the pairs of subalgebras of the form $(\mf{a},\mf{b})$ with
\begin{align*}
\mf{a} &= \biggl ( \bigoplus_{\mu} \mf{gl} (k_{\mu},\C) \biggr ) \oplus \biggl ( \bigoplus_{ \nu } \mf{so} (k_{\nu},\C) \biggr ) \oplus \biggl ( \bigoplus_{\lambda} \mf{sp}(2 k_{\lambda},\C) \biggr ) , \\
\mf{b} &= \biggl ( \bigoplus_{\mu} \mf{gl} (\ell_{\mu},\C ) \biggr ) \oplus \biggl ( \bigoplus_{ \nu } \mf{sp} (2 \ell_{\nu},\C) \biggr ) \oplus \biggl ( \bigoplus_{\lambda} \mf{so}(\ell_{\lambda},\C) \biggr ),
\end{align*}
where $k_{\gamma} = 0$ if and only if $\ell_{\gamma} = 0$, where $n=\sum_{\mu} k_{\mu}\ell_{\mu} + \sum_{\nu} k_{\nu} \ell_{\nu} + \sum_{\lambda} k_{\lambda} \ell_{\lambda}$, and where $2 \not \in \{ k_{\nu} \}_{\nu} \cup \{ \ell_{\lambda} \}_{\lambda}$.
\end{thm}

\begin{proof}
We start by showing that the pairs of subalgebras of the form $(\mf{a}, \mf{b})$ described in the proposition statement are indeed dual pairs. To this end, suppose that $(G_1, G_2)$ is a dual pair in $Sp(2n,\C)$ with $2 \not \in \{ k_{\nu} \}_{\nu} \cup \{ \ell_{\lambda} \}_{\lambda}$ (using the notation of Theorem \ref{thm:dual-pairs-in-Sp}). Set $\mf{g}_i := \Lie (G_i)$. We claim that $Z_{Sp(2n,\C)}(\langle \exp \mf{g}_i \rangle ) = Z_{Sp(2n,\C)} (G_i)$. To see this, recall first that $\exp \mf{gl}(k_{\mu},\C) = GL(k_{\mu},\C)$. Next, since $SO(k_{\nu},\C)$ and $Sp(2k_{\lambda},\C)$ are connected, \cite[Theorem 4.6(c)]{sepanski} gives that $\langle \exp \mf{so} (k_{\nu},\C) \rangle = SO(k_{\nu},\C)$ and that $\langle \exp \mf{sp}(2k_{\lambda},\C) \rangle = Sp(2k_{\lambda},\C)$. Finally, since the standard representations of $SO(k_{\nu},\C)$ and $O(k_{\nu},\C)$ are both irreducible for $k_{\nu}\neq 2$ (by Lemma \ref{lem:standard-irreps}), we have that $Z_{Sp(2n,\C)} ( SO(k_{\nu},\C) ) = Z_{Sp(2n,\C)} (O(k_{\nu},\C) )$. Therefore, Proposition \ref{prop:1} and Theorem \ref{thm:dual-pairs-in-Sp} give that the pairs of subalgebras of the form $(\mf{a}, \mf{b})$ are dual pairs.

Conversely, suppose that $(\mf{g}_1, \mf{g}_2)$ is a dual pair in $\mf{sp}(2n,\C)$. Set $G_2 = Z_{Sp(2n,\C)} ( \langle \exp \mf{g}_1 \rangle  )$, and $G_1 = Z_{Sp(2n,\C)}(G_2)$. Then by Proposition \ref{prop:2}, $(G_1, G_2)$ is a dual pair in $Sp(2n,\C)$ with $\Lie (G_1) = \mf{g}_1$ and $\Lie (G_2) = \mf{g}_2$. Therefore, $G_1$ and $G_2$ can be written as
\begin{align*}
G_1 &= \prod_{\mu} GL(k_{\mu},\C) \prod_{\nu} O(k_{\nu},\C) \prod_{\lambda} Sp(2k_{\lambda},\C),  \\
G_2 &= \prod_{\mu} GL(\ell_{\mu},\C) \prod_{\nu} Sp(2\ell_{\nu},\C) \prod_{\lambda} O(\ell_{\lambda},\C) ,
\end{align*}
and $\mf{g}_1$ and $\mf{g}_2$ can be written as 
\begin{align*}
\mf{g}_1 &= \biggl( \bigoplus_{\mu} \mf{gl} (k_{\mu},\C) \biggr) \oplus \biggl ( \bigoplus_{ \nu } \mf{so} (k_{\nu},\C) \biggr) \oplus \biggl( \bigoplus_{\lambda} \mf{sp}(2 k_{\lambda},\C) \biggr) , \\
\mf{g}_2 &= \biggl( \bigoplus_{\mu} \mf{gl} (\ell_{\mu},\C ) \biggr) \oplus \biggl( \bigoplus_{ \nu } \mf{sp} (2 \ell_{\nu},\C) \biggr) \oplus \biggl( \bigoplus_{\lambda} \mf{so}(\ell_{\lambda},\C) \biggr),
\end{align*}
where $k_{\gamma} = 0$ if and only if $\ell_{\gamma} = 0$, and where $\sum_{\mu} k_{\mu}\ell_{\mu} + \sum_{\nu} k_{\nu} \ell_{\nu} + \sum_{\lambda} k_{\lambda} \ell_{\lambda} = n$. It therefore remains to show that $2 \not \in \{ k_{\nu} \}_{\nu} \cup \{ \ell_{\nu} \}_{\nu}$. We have that
$$ \langle \exp \mf{g}_1 \rangle = \prod_{\mu} GL(k_{\mu},\C) \prod_{\nu} SO(k_{\nu},\C) \prod_{\lambda} Sp(2k_{\lambda},\C).$$ 
Suppose that $k_{\nu_0} = 2$. Then Lemma \ref{lem:standard-irreps} does not apply, as the standard representation of $SO(2,\C)$ is not irreducible. In fact, we claim that $Z_{Sp(2k_{\nu_0}\ell_{\nu_0},\C)} ( SO(k_{\nu_0},\C) ) \supsetneq Sp(2\ell_{\nu_0},\C)$. By Theorem \ref{thm:dual-pairs-in-Sp}, 
$$Z_{Sp(2k_{\nu_0}\ell_{\nu_0},\C)} ( SO(k_{\nu_0},\C) ) \supseteq Z_{Sp(2k_{\nu_0}\ell_{\nu_0},\C)} ( O(k_{\nu_0},\C) ) = Sp(2\ell_{\nu_0},\C),$$
where $Sp(2\ell_{\nu_0},\C)$ is embedded in $Sp(2k_{\nu_0}\ell_{\nu_0},\C)$ as 
$$\left \{ \begin{pmatrix} X & \\ & X \end{pmatrix} \; \Big \vert \; X \in Sp(2\ell_{\nu_0},\C)  \right \}.$$
However, $Z_{Sp(2k_{\nu_0}\ell_{\nu_0},\C)} ( SO(k_{\nu_0},\C) )$ contains
$$\left \{ \begin{pmatrix} & Y \\ -Y & \end{pmatrix} \; \Big \vert \; Y \in Sp(2\ell_{\nu_0},\C) \right \}.$$ 
In this way, we see that
$$Z_{Sp(2n,\C)} (  \langle \exp \mf{g}_1 \rangle  ) \supsetneq G_2,$$
which contradicts Proposition \ref{prop:2}. It follows that $2 \not \in \{ k_{\nu} \}_{\nu} \cup \{ \ell_{\nu} \}_{\nu}$, completing the proof.
\end{proof}

\section{Dual pairs in \texorpdfstring{$O(n,\C)$}{O(n,C)}, \texorpdfstring{$SO(n,\C)$}{SO(n,C)}, \texorpdfstring{$PO(n,\C)$}{PO(n,C)}, \texorpdfstring{$PSO(n,\C)$}{PSO(n,C)}, and \texorpdfstring{$\mf{so}(n,\C)$}{so(n,C)} } \label{sec:o}

\subsection{Dual pairs in \texorpdfstring{$O(n,\C)$}{O(n,C)}} \label{subsec:O}

The classification of dual pairs in $O(n,\C)$ follows from an analysis quite similar to that for $Sp(2n,\C)$ in Section \ref{sec:sp}. As a result, we omit some of the proofs in this section. 

\begin{thm} \label{thm:dual-pairs-in-O}
The dual pairs of $O(n,\C)$ are exactly the pairs of groups of the form
\begin{equation*}
    \left ( \prod_{\mu} GL(k_{\mu},\C) \prod_{ \nu } O(k_{\nu},\C) \prod_{\lambda} Sp(2k_{\lambda},\C), \; \prod_{\mu} GL(\ell_{\mu},\C) \prod_{ \nu } O(\ell_{\nu},\C) \prod_{\lambda} Sp(2\ell_{\lambda},\C) \right ),
\end{equation*}
where $k_{\gamma} = 0$ if and only if $\ell_{\gamma} = 0$, and where $2\sum_{\mu} k_{\mu} \ell_{\mu} + \sum_{\nu} k_{\nu}\ell_{\nu} + 4\sum_{\lambda} k_{\lambda} \ell_{\lambda} = n$.
\end{thm}

\begin{proof}
     This follows from similar arguments as in the proof of Theorem \ref{thm:dual-pairs-in-Sp}.
\end{proof}

\subsection{Dual pairs in \texorpdfstring{$SO(n,\C)$}{SO(n,C)}} \label{subsec:SO}

In this subsection, we classify the dual pairs in $SO(n,\C)$ and in $\mf{so}(n,\C)$, both of which will depend on the parity of $n$.

\begin{thm} 
When $n$ is odd, the dual pairs of $SO(n,\C)$ are exactly the pairs of groups 
\begin{align*}
    A&= \left ( \prod_{\mu} GL(k_{\mu},\C)  \prod_{ \nu } O(k_{\nu},\C) \prod_{\lambda} Sp(2k_{\lambda},\C) \right  ) \cap SO(n,\C) ,\\
    B&= \left ( \prod_{\mu} GL(\ell_{\mu},\C)  \prod_{ \nu } O(\ell_{\nu},\C) \prod_{\lambda} Sp(2\ell_{\lambda},\C) \right  ) \cap SO(n,\C),
\end{align*}
where $k_{\gamma} = 0$ if and only if $\ell_{\gamma} = 0$, and where $n=2\sum_{\mu} k_{\mu}\ell_{\mu} + \sum_{\nu} k_{\nu} \ell_{\nu} + 4\sum_{\lambda} k_{\lambda} \ell_{\lambda}$. In other words, when $n$ is odd, the dual pairs of $SO(n,\mathbb{C})$ are in bijection with the dual pairs of $O(n,\mathbb{C})$. 
\end{thm}

\begin{proof}
This follows from Theorem \ref{thm:subgroup-bijection} and the observation that $O(n,\mathbb{C}) = Z \cdot SO(n,\mathbb{C})$ when $n$ is odd, where $Z$ is the center of $O(n,\mathbb{C})$ (i.e. $Z = \{ \pm I_n \}$). First, recall that any orthogonal matrix has determinant $\pm 1$. Given a matrix $g \in O(n,\C)$, if $\det g = 1$, then there is nothing to show. On the other hand, if $\det g = -1$, then we can write $g = (-I_n)g(-I_n)$; since $n$ is odd, we see that $g(-I_n) \in SO(n,\C)$, and this completes the proof.
\end{proof}

\begin{thm} \label{thm:SO-even}
When $n$ is even, the dual pairs of $SO(n,\C)$ are exactly the pairs of groups 
\begin{align*}
    A&= \left ( \prod_{\mu} GL(k_{\mu},\C)  \prod_{ \nu } O(k_{\nu},\C) \prod_{\lambda} Sp(2k_{\lambda},\C) \right  ) \cap SO(n,\C) ,\\
    B&= \left ( \prod_{\mu} GL(\ell_{\mu},\C)  \prod_{ \nu } O(\ell_{\nu},\C) \prod_{\lambda} Sp(2\ell_{\lambda},\C) \right  ) \cap SO(n,\C),
\end{align*}
where $k_{\gamma} = 0$ if and only if $\ell_{\gamma} = 0$, where $n=2\sum_{\mu} k_{\mu}\ell_{\mu} + \sum_{\nu} k_{\nu} \ell_{\nu} + 4\sum_{\lambda} k_{\lambda} \ell_{\lambda}$, and where either $2 \not \in \{ k_{\nu} \}_{\nu} \cup \{ \ell_{\nu} \}_{\nu}$ or $\vert \{ \nu \} \vert \geq 2$. 
\end{thm}

\begin{proof}
In this proof, we drop the $\C$'s when describing the complex classical groups so as to make the expressions more readable. By Remark \ref{rmk:all pairs}, the dual pairs in $SO(n)$ are exactly the pairs of the form 
$$( Z_{SO(n)} (S) , \; Z_{SO(n)} ( Z_{SO(n)} (S) ) ),$$
where $S \subseteq SO(n)$. By Theorem \ref{thm:dual-pairs-in-O}, $Z_{SO(n)} (S)$ is of the form 
\begin{align*}
    Z_{SO(n)} (S) &= Z_{O(n)} ( S ) \cap SO(n)  \\
    &= \left ( \prod_{\mu} GL(k_{\mu}) \prod_{ \nu } O(k_{\nu}) \prod_{\lambda} Sp(2k_{\lambda}) \right ) \cap SO(n) \\
    &= \prod_{\mu} GL(k_{\mu}) \times \left ( \prod_{ \nu } O(k_{\nu}) \cap SO(\textstyle \sum_\nu k_\nu \ell_\nu ) \right ) \times \prod_{\lambda} Sp(2k_{\lambda})
\end{align*}
for some $\{ k_\gamma, \ell_\gamma \}_\gamma$ such that $n=2\sum_{\mu} k_{\mu}\ell_{\mu} + \sum_{\nu} k_{\nu} \ell_{\nu} + 4\sum_{\lambda} k_{\lambda} \ell_{\lambda}$ and such that $k_{\gamma} = 0$ if and only if $\ell_{\gamma} = 0$. For notational convenience, set $m=\sum_\nu k_\nu \ell_\nu$. Note that if $2\in \{ k_{\nu} \}_{\nu}$ and $\vert \{ \nu \} \vert =1$, then $\prod_\nu O(k_\nu) \cap SO(m)$ becomes $SO(2) \simeq GL(1)$, which can be analyzed along side the other $GL$ factors. We can therefore assume that we don't have both $2\in \{ k_{\nu} \}_{\nu}$ and $\vert \{ \nu \} \vert =1$. Now, note that 
$$    Z_{SO(n)} (  Z_{SO(n)} (S) ) = \prod_\mu GL(\ell_\mu) \times Z_{SO(m)} \left ( \prod_{ \nu } O(k_{\nu}) \cap SO(m) \right ) \times \prod_\lambda Sp(2\ell_\lambda)$$
and that
$$Z_{SO(m)} \left ( \prod_{ \nu } O(k_{\nu}) \cap SO(m) \right ) = Z_{O(m)} \left ( \prod_{ \nu } O(k_{\nu}) \cap SO(m) \right ) \cap SO(m).$$
Moreover, we have that
\begin{align*}
    Z_{O(m)} \left ( \prod_{ \nu } O(k_{\nu}) \cap SO(m) \right ) &= Z_{O(m)} \left ( \left ( \prod_{ \nu \colon k_\nu = 2 } O(2) \times \prod_{ \nu \colon k_\nu \neq 2 } O(k_\nu) \right ) \cap SO(m) \right ) \\
    &\subseteq Z_{O(m)} \left ( \prod_{ \nu \colon k_\nu = 2 } SO(2) \times \prod_{ \nu \colon k_\nu \neq 2 } SO(k_\nu) \right )\\
    &\simeq Z_{O(m)} \left ( \prod_{ \nu \colon k_\nu = 2 } GL(1) \times \prod_{ \nu \colon k_\nu \neq 2 } SO(k_\nu) \right ) \\
    &= \prod_{ \nu \colon k_\nu = 2 } GL(\ell_\nu) \times \prod_{ \nu \colon k_\nu \neq 2 } O(\ell_\nu).
\end{align*}
Now, since we are assuming that either $2\not \in \{k_\nu \}_\nu$ or $\vert \{ \nu \} \vert \geq 2$, we get from the above that
\begin{equation} \label{eq:O(m)-centralizer}
    Z_{O(m)} \left ( \prod_{ \nu } O(k_{\nu}) \cap SO(m) \right ) = \prod_{ \nu } O(\ell_{\nu}).
\end{equation}
When $2\not \in \{k_\nu \}_\nu$, this is clear. To see why this is true when $2 \in \{k_\nu \}_\nu$ and $\vert \{ \nu \} \vert \geq 2$, note that under the change of basis giving $SO(2) \simeq GL(1)$, we have that $SO(2)$, $O(2)$, and $GL(\ell_\nu)$ embed into $O(2\ell_\nu)$ as 
\begin{align*}
    SO(2) &= \left \{ \begin{pmatrix} x \cdot I_{\ell_\nu} &\\& x^{-1} \cdot I_{\ell_\nu} \end{pmatrix} \; : \; x \in \C^{\times} \right \}, \\
    O(2) &= \left \{ \begin{pmatrix} x \cdot I_{\ell_\nu} &\\& x^{-1} \cdot I_{\ell_\nu} \end{pmatrix} \; : \; x \in \C^{\times} \right \} \cup \left \{ \begin{pmatrix} & y \cdot I_{\ell_\nu} \\ y^{-1} \cdot I_{\ell_\nu} & \end{pmatrix} \; : \; y \in \C^{\times} \right \} ,\\
    GL(\ell_\nu) &= \left \{ \begin{pmatrix} g & \\ & (g^t)^{-1} \end{pmatrix} \; : \; g \in GL(\ell_\nu) \right \},
\end{align*}
where the orthogonal form on $O(2\ell_\nu)$ is $\left ( \begin{smallmatrix} & I_{\ell_\nu} \\ I_{\ell_\nu} & \end{smallmatrix} \right )$. Now, when $2 \in \{k_\nu \}_\nu$ and $\vert \{ \nu \} \vert \geq 2$, for any $\nu$ with $k_\nu = 2$ there exists an element of $\prod_{ \nu } O(k_{\nu}) \cap SO(m)$ with $\nu$-factor of the form $\left ( \begin{smallmatrix} & y \cdot I_{\ell_\nu} \\ y^{-1} \cdot I_{\ell_\nu} \end{smallmatrix} \right )$. With this in mind, it is not hard to verify that $\left ( \begin{smallmatrix} g & \\ & (g^t)^{-1} \end{smallmatrix} \right ) \in GL(\ell_\nu)$ centralizes all of the possible $\nu$-blocks if and only if $g \in O(\ell_\nu)$. This in turn verifies \eqref{eq:O(m)-centralizer}. 

It follows that $ Z_{SO(n)} (  Z_{SO(n)} (S) )$ is of the form  
\begin{equation} \label{eq:SO(n)-centralizer}
\left ( \prod_{\mu} GL(\ell_{\mu})  \prod_{ \nu } O(\ell_{\nu}) \prod_{\lambda} Sp(2\ell_{\lambda}) \right  ) \cap SO(n)   
\end{equation}
whenever $2 \not \in \{k_\nu \}_\nu$ or $\vert \{ \nu \} \vert \geq 2$. Finally, one can re-trace the above argument to check whether \eqref{eq:SO(n)-centralizer} is indeed a member of a dual pair. When doing so, we gain the requirement that either $2\not \in \{\ell_\nu \}_\nu$ or $\vert \{ \nu \} \vert \geq 2$.
\end{proof}

\begin{remark}
    Note that when $n=2\sum_{\mu} k_{\mu}\ell_{\mu} + \sum_{\nu} k_{\nu} \ell_{\nu} + 4\sum_{\lambda} k_{\lambda} \ell_{\lambda}$ is odd, the condition that either $2 \not \in \{k_\nu\}_\nu \cup \{\ell_\nu \}_\nu$ or $\vert \{\nu \}\vert \geq 2$ is automatically satisfied. (Thus, the statement of Theorem \ref{thm:SO-even} actually holds for both $n$ odd and $n$ even.)
\end{remark}

\subsection{Dual pairs in \texorpdfstring{$PO(n,\C)$}{PO(n,C)}} \label{subsec:PO}

Let $p\colon O(n,\C) \rightarrow PO(n,\C)$ be the canonical projection.

\begin{prop}
Define
$$\displaystyle \mathcal{O}(m,\C) := \left \langle O(m,\C), \; \begin{pmatrix} & I_{m} \\ -I_m & \end{pmatrix} \right \rangle \subseteq O(2m,\C).$$
Let $(G_1, G_2)$ be the $O(n,\C)$ dual pair
\begin{equation*}
  \left ( \prod_{\mu} GL(k_{\mu},\C) \prod_{ \nu } O(k_{\nu},\C) \prod_{\lambda} Sp(2k_{\lambda},\C), \; \prod_{\mu} GL(\ell_{\mu},\C) \prod_{ \nu } O(\ell_{\nu},\C) \prod_{\lambda} Sp(2\ell_{\lambda},\C) \right ),
\end{equation*}
where $k_{\gamma} = 0$ if and only if $\ell_{\gamma} = 0$, and where $2\sum_{\mu} k_{\mu} \ell_{\mu} + \sum_{\nu} k_{\nu}\ell_{\nu} + 4\sum_{\lambda} k_{\lambda} \ell_{\lambda} = n$. 
\begin{enumerate}[label = (\roman*)]
\item If $n$ is even, then $(p (O(2,\C)), p (\mathcal{O}(\frac{n}{2},\C)))$ is a dual pair in $PO( n,\C )$;
\item If $2 \not \in \{ k_{\nu} \}_{\nu} \cup \{ \ell_{\nu} \}_{\nu}$, then $(p(G_1), p(G_2))$ is a dual pair in $PO(n,\C)$;
\item If $\vert \{ \mu , \nu , \lambda \} \vert \geq 2$, then $(p(G_1), p(G_2))$ is a dual pair in $PO(n,\C)$.
\end{enumerate}
\end{prop}

\begin{proof}
    This follows from similar arguments as in the proof of Proposition \ref{prop:dual-pairs-in-PSp}.
\end{proof}

Classifying the remaining dual pairs in $PO(n,\C)$ appears to be a difficult question (related to the problem of classifying the disconnected dual pairs in $PGL(n,\C)$) and is beyond the scope of this paper.

\subsection{Dual pairs in \texorpdfstring{$PSO(n,\C)$}{PSO(n,C)}} \label{subsec:PSO}

Let $p\colon SO(n,\C) \rightarrow PSO(n,\C)$ be the canonical projection.

\begin{prop} \label{prop:PSO}
    Let $(G_1, G_2)$ be the $SO(n,\C)$ dual pair 
    \begin{align*}
    G_1&= \left ( \prod_{\mu} GL(k_{\mu},\C)  \prod_{ \nu } O(k_{\nu},\C) \prod_{\lambda} Sp(2k_{\lambda},\C) \right  ) \cap SO(n,\C) ,\\
    G_2&= \left ( \prod_{\mu} GL(\ell_{\mu},\C)  \prod_{ \nu } O(\ell_{\nu},\C) \prod_{\lambda} Sp(2\ell_{\lambda},\C) \right  ) \cap SO(n,\C),
\end{align*}
where $k_{\gamma} = 0$ if and only if $\ell_{\gamma} = 0$, where $n=2\sum_{\mu} k_{\mu}\ell_{\mu} + \sum_{\nu} k_{\nu} \ell_{\nu} + 4\sum_{\lambda} k_{\lambda} \ell_{\lambda}$, and where either $2 \not \in \{ k_{\nu} \}_{\nu} \cup \{ \ell_{\nu} \}_{\nu}$ or $\vert \{ \nu \} \vert \geq 2$. Then $(p(G_1),p(G_2))$ is a dual pair in $PSO(n,\C)$
\end{prop}

\begin{proof}
    This follows from similar arguments as in the proof of Proposition \ref{prop:dual-pairs-in-PSp}.
\end{proof}

\begin{remark}
     Note that if $n$ is odd, then $Z(SO(n,\C)) =\{ 1 \}$, meaning $PSO(n,\C) \simeq SO(n,\C) $. Therefore, when $n$ is odd, Proposition \ref{prop:PSO} gives a complete classification of $PSO(n,\C)$ dual pairs. However, when $n$ is even, classifying the remaining dual pairs in $PSO(n,\C)$ again appears to be a difficult question, and is beyond the scope of this paper.
\end{remark}

\subsection{Dual pairs in \texorpdfstring{$\mf{so}(n,\C)$}{so(n,C)}} \label{subsec:so-LA}

Recall that $O(n,\C)$ and $SO(n,\C)$ have the same Lie algebra (i.e.~$\mf{o}(n,\C)=\mf{so}(n,\C)$).

\begin{thm}
The dual pairs of $\mf{so}(n,\C)$ are exactly the pairs of subalgebras of the form $(\mf{a},\mf{b})$ with
\begin{align*}
\mf{a} &= \biggl( \bigoplus_{\mu} \mf{gl} (k_{\mu},\C) \biggr) \oplus \biggl( \bigoplus_{ \nu } \mf{so} (k_{\nu},\C) \biggr) \oplus \biggl( \bigoplus_{\lambda} \mf{sp}(2 k_{\lambda},\C) \biggr) , \\
\mf{b} &= \biggl( \bigoplus_{\mu} \mf{gl} (\ell_{\mu},\C ) \biggr) \oplus \biggl( \bigoplus_{ \nu } \mf{so} ( \ell_{\nu},\C) \biggr) \oplus \biggl( \bigoplus_{\lambda} \mf{sp}(2\ell_{\lambda},\C) \biggr),
\end{align*}
where $k_{\gamma} = 0$ if and only if $\ell_{\gamma} = 0$, where $n=2\sum_{\mu} k_{\mu}\ell_{\mu} + \sum_{\nu} k_{\nu} \ell_{\nu} + 4\sum_{\lambda} k_{\lambda} \ell_{\lambda}$, and where $2 \not \in \{ k_{\nu} \}_{\nu} \cup \{ \ell_{\nu} \}_{\nu}$. 
\end{thm}

\begin{proof}
    This follows from similar arguments as in the proof of Theorem \ref{thm:dual-pairs-in-sp-LA}.
\end{proof}

\section*{Acknowledgements}

The author would like to thank David Vogan for suggesting this topic of study and for his guidance throughout the project. She would also like to thank Calder Morton-Ferguson for his helpful comments.

\bibliography{biblio} 
\bibliographystyle{plain}

\end{document}